\documentclass[final]{siamltex}
\usepackage{amssymb,amsmath}
\usepackage{float,epsfig}
\usepackage{algpseudocode}
\usepackage{algorithm}
\usepackage{enumitem}
\usepackage{array}
\usepackage{tikz}
\usepackage{csquotes}
\usepackage{url}
\usepackage[
    a4paper,        
    left=1in,      
    right=1in,     
    top=1in,       
    bottom=1in,    
    includehead,    
    includefoot     
]{geometry}

\newcommand{\Arrow}[1]{%
	\parbox{#1}{\tikz{\draw[->](0,0)--(#1,0);}}
}

\newcommand{\rsss}{\rotatebox[]{90}{$\boxminus$}\kern-0.7em{\mathrel{\raisebox{.2ex}{\Arrow{.35cm}}}}\!\!}

\newcommand{\csss}{\text{$\boxminus\kern-0.655em{\mathrel{\raisebox{-.2ex}{\rotatebox[]{-90}{\Arrow{.34cm}}}}}$\ }}

\newtheorem{remark}[theorem]{ Remark}
\newtheorem{exam}[theorem]{\bf Example}
\newtheorem{observation}[theorem]{ Observation}
\newcommand{\ba}{\begin{array}}
\newcommand{\ea}{\end{array}}

\newcommand{\be}{\begin{equation}}
\newcommand{\ee}{\end{equation}}
\newcommand{\beano}{\begin{eqnarray*}}
\newcommand{\eeano}{\end{eqnarray*}}

\def\diag{\mathrm{diag}}

\usepackage[us,12hr]{datetime}

\title{Arithmetical Structures on Ladder Graphs }
\author{Namita Behera \thanks{Department of Mathematics, Sikkim University, Sikkim-737102, India, ({\tt nbehera@cus.ac.in, niku.namita@gmail.com}) } \and Dilli Ram Chhetri  \thanks{Department of Mathematics, Sikkim University, Sikkim-737102, India, ({\tt drchhetri.22pdmt01@sikkimuniversity.ac.in, dilluchhetri1@gmail.com}).} \and Raj Bhawan Yadav \thanks{Department of Mathematics, Sikkim University, Sikkim-737102, India ({\tt rbyadav01@cus.ac.in})}
 }

\begin{document}

\maketitle

\begin{abstract}
In this paper, we investigate arithmetical structures on Cartesian product graphs, particularly, ladder graph of the form $P_2 \square P_m$ and grid graph of the form $P_n \square P_m$.  An arithmetical structure on a finite and connected graph $G$ is a pair $(\textbf{d}, \textbf{r})$ of positive integer vectors such that $\textbf{r}$ is primitive (the gcd of its entries is $1$) and $(\diag(\textbf{d}) - A)\textbf{r} = 0,$ where $A$ is the adjacency matrix of $G.$ Arithmetical structures have been widely studied for basic graph families such as paths and cycles. Extending these ideas to graph products, we first analyze the ladder graph $P_2 \square P_m$, deriving structural properties and identifying patterns in the corresponding arithmetical configurations. We then generalize these results to the grid graph $P_n \square P_m$, where increased complexity arises due to higher-dimensional interactions. Our work provides new insights into the behavior, characterization, and enumeration of arithmetical structures on grid-like graphs, contributing to the broader understanding of Laplacian-based invariants and their combinatorial properties.
\end{abstract}



\section{Introduction}
The study of graph products occupies a central position in graph theory, offering a powerful framework for constructing complex graphs from simpler components. Among various graph products, the Cartesian product has been extensively investigated due to its rich structural properties and wide-ranging applications in combinatorics, network theory, and discrete optimization. Given two graphs $G$ and $H$, their Cartesian product, denoted by $G \square H$, is defined as the graph with vertex set $V(G) \times V(H)$, where two vertices $(g_1, h_1)$ and $(g_2, h_2)$ are adjacent if and only if either $g_1 = g_2$ and $h_1h_2 \in E(H)$, or $h_1 = h_2$ and $g_1g_2 \in E(G)$ \cite{CartesianProductofGraphs}. This operation plays a fundamental role in understanding higher-dimensional graph structures. Further, this construction preserves many structural properties of the factor graphs while introducing new combinatorial features

In parallel, the notion of arithmetical structures on graphs has emerged as an important topic connecting combinatorial graph theory with algebraic and number theoretic ideas. Arithmetical structures on graphs were introduced and systematically studied in \cite{HCEV18,BHDNJC18,arithmeticaloncompletegraphs,PathwithDoubleedgeArithstructure,keyesReiter2021boundingthenumberofArithstructure,LD89,KAAL20,BANBDCRBY2024,CHVC18,chhetri2025arithmeticalonFanGraph,criticalpolynomialofgraphbyLorenzini}, where they arise naturally in connection with Laplacian matrices and determinantal ideals.

Let $G = (V,E)$ be a finite and connected graph, where $V$ is the vertex set and 
$E \subseteq V \times V$ represents the edge set. An \emph{arithmetical structure} 
on $G$ is defined as a pair $(\mathbf{d}, \mathbf{r})$, where $\mathbf{d}$ and 
$\mathbf{r}$ are positive integer vectors, $\mathbf{r}$ is \emph{primitive} 
(i.e., the greatest common divisor of its entries is $1$), and
\[
(\operatorname{diag}(\mathbf{d}) - A)\mathbf{r} = 0,
\]
where $A$ is the adjacency matrix of $G$. The adjacency matrix 
$A = [a_{ij}]$ is a symmetric matrix defined by
\[
a_{ij} =
\begin{cases}
1, & \text{if } (i,j) \in E, \\
0, & \text{otherwise},
\end{cases}
\quad \text{for } i,j \in V.
\]

This definition extends the concept of the Laplacian arithmetical structure, 
where $\mathbf{d}$ represents the vector of vertex degrees and 
$\mathbf{r} = \mathbf{1} = (1,1,\dots,1)$. Note that $\mathbf{d}$ and 
$\mathbf{r}$ uniquely determine each other, allowing us to regard $\mathbf{d}$, 
$\mathbf{r}$, or the pair $(\mathbf{d}, \mathbf{r})$ as an arithmetical structure 
on $G$. To avoid ambiguity, we will sometimes refer to these as 
\emph{arithmetical $d$-structures} or \emph{arithmetical $r$-structures}~\cite{BANBDCRBY2024}

The set of all arithmetical structures on $G$ is denoted by 
$\operatorname{Arith}(G)$ and sometimes by $\mathcal{A}(G)$. 
The data $(G,\mathbf{d},\mathbf{r})$ together define an 
\emph{arithmetical graph}. We denote the set of $r$-structures of $G$ by $\mathcal{A}_{r}(G)$.
These structures are closely related to the critical group of a graph \cite{LD89,HCEV18,BHDNJC18,criticalgroupon(starandcompletegraph),AAJL18,SandpilegroupoverEllpiticCurves,archer2023},  chip-firing games \cite{Chipfiringgamesongraph1991,biggs1999,CP18,Corry2018DivisorsAS,archer2025}, and divisor theory on graphs, and they provide insight into the interplay between graph structure and integer valued functions.

Significant progress has been made in understanding arithmetical structures on specific families of graphs. For instance, the complete classification for path graphs $P_n$ is known \cite{braun2017arithmetical} and the number of arithmetical structures on $P_n$ is given by the Catalan numbers \cite{catalannumberbycarlitz}, revealing a deep connection with classical combinatorial objects. Similar studies have been carried out for cycles and certain classes of trees, where recursive and combinatorial techniques play a central role.

While arithmetical structures have been extensively studied for basic graph families such as paths, cycles, wheel graphs, Fan graphs, and trees \cite{lionelleivnewhatisasandpile, BANBDCRBY2024, chhetri2025arithmeticalonFanGraph},  comparatively less attention has been given to their behavior under graph products. In particular, Cartesian products of path graphs offer a natural and tractable setting for extending these ideas. The graph $P_2 \square P_m$, commonly known as a ladder graph, represents the simplest nontrivial product structure, while the more general graph $P_n \square P_m$ corresponds to a rectangular grid graph. These grid-like structures exhibit a high degree of regularity and symmetry, making them suitable for exploring how local arithmetical constraints propagate across larger networks.

The investigation of arithmetical structures on $P_2 \square P_m$ and $P_n \square P_m$ reveals new combinatorial phenomena arising from the interaction between the two path components. In ladder graphs, the relatively simple structure allows explicit constructions and recursive descriptions of arithmetical structures. However, as one moves to general grid graphs, additional complexity emerges due to the increased number of adjacency relations and the multidimensional nature of the graph. This makes the study both challenging and rich in structure.

The aim of this paper is to analyze the existence, characterization, and enumeration of arithmetical structures on the Cartesian product graphs $P_2 \square P_m$ and $P_n \square P_m$. By first examining the ladder case and then extending the results to general grid graphs, we seek to understand how arithmetical properties evolve under graph products and to contribute to the broader theory connecting Laplacian-based invariants with combinatorial graph constructions. Unlike paths, these graphs contain cycles and exhibit a richer combinatorial structure. We develop a structural framework to study arithmetical structures on these graphs, including a decomposition into symmetric and non-symmetric components. This approach allows us to relate certain classes of arithmetical structures on $P_2 \square P_m$ to those on $P_m$, while also identifying new families that arise purely from the product structure.

Our results provide a first step toward understanding arithmetical structures on product graphs and suggest further directions for studying more general Cartesian products such as $P_m \square P_n$.

The paper is organized as follows. In section~2, we present the necessary preliminaries. 
Section~3 is devoted to the study of arithmetical structures on ladder graphs, specifically the construction of $\operatorname{Arith}(P_2 \square P_m)$ from $\mathcal{A}(P_m)$. In section~4, we investigate the arithmetical structures on $P_2 \square P_m$ via $C_4$-decomposition. In addition, we study arithmetical structures on grid graph $P_n \square P_m$. Finally, Section~5 presents the conclusion.

\section{Preliminaries}
In this section, we present some fundamental definitions and results that will be used throughout the paper.
\begin{definition}
A square matrix $A$ is called reducible if there exists a permutation matrix $P$ such that $$PAP^{t} = \begin{pmatrix}
    A_1 & * \\
    0 & A_2 \\
\end{pmatrix}$$ for some non-trivial square matrices $A_1$ and $A_2.$    
\end{definition}

\begin{definition} [Neighbourhood of a vertex v in a graph G] \cite{west2001introduction} 
 The neighbourhood of a vertex v in a graph G is the subgraph of G induced by all vertices adjacent to v, i.e., the graph composed of the vertices adjacent to v and all edges connecting vertices adjacent to v.  The neighbourhood is often denoted by $N_G(v)$.
\end{definition}

Here we recall the classical concept of an $M$-matrix and study a class of $M$-matrices whose proper principal minors are positive and its determinant is non-negative. Let us begin with some definitions: \cite{HCEV18}

\begin{definition}[Real non-negative matrix] \cite{HCEV18}
A real square matrix is called non-negative if all its entries are non-negative real numbers.    
\end{definition}

\begin{definition}[$Z$-matrix] \cite{HCEV18}
 A real matrix $ A = (a_{i,j}) \in \mathbb{R}^{n \times n}$ is called a Z-matrix if $a_{i,j} \leq 0 ,\ for \ all \ i \neq j$. 
\end{definition}

\begin{definition}[$M$-matrix] \cite{HCEV18}
 A $Z$-matrix $A$ is an $M$-matrix if there exists a non-negative matrix $N$ and a non-negative number $\alpha$ such that such that $A = \alpha I- N \ and \ \alpha \geq \rho(N)$, where $\rho(N) = max\{|\lambda|  : \lambda \in \sigma(N)\}$.   
\end{definition}
The study of $M$-matrices can be divided into two major parts: non-singular M-matrices and singular $M$-matrices (see[\cite{Berman1994andPlemmons}, Section 6.2 and 6.4])

\begin{definition}
 \cite{HCEV18} A real matrix $A =(a_{i,j})$ is called an almost non-singular $M$-matrix if $A$ is a $Z$-matrix, all its proper principal minors are positive and its determinant is non-negative.    
\end{definition}

\begin{theorem} \cite{HCEV18}\label{Theorem 2.6 from Corrales and Valencia}
 If $M$ is a real $Z$-matrix, then the following conditions are equivalent: 
 \begin{itemize}
     \item [(1)] $M$ is an almost non-singular $M$-matrix.
     \item [(2)] $M +D$ is a non-singular $M$matrix for any diagonal matrix $D>_{\neq} 0$. 
     \item [(3)] $\det(M) \geq 0$ and $\det(M +D) {\gneq} \det(M +D'){\gneq} 0$  for any diagonal matrices $D >_{\neq} D'{\gneq} 0.$
 \end{itemize} 
\end{theorem}

\begin{theorem} \cite{HCEV18}\label{Theorem 3.2 from Corrales and Valencia}
Let $M$ be a $Z$-matrix. If there exists ${\bf r}$ with all its entries positive such that $M{\bf r}^{t} = 0^t,$ then $M$ is an $M$-matrix. Moreover, $M$ is an almost non-singular $M$-matrix with $\det(M) = 0$ if and only if $M$ is irreducible and there exists ${\bf r}$ with all its entries positive such that $M{\bf r}^t = 0^t. $   
\end{theorem}

\begin{definition}
Let $A = [a_{ij}] \in \mathbb{R}^{m \times n}$ and 
$B \in \mathbb{R}^{p \times q}$. The \emph{Kronecker product} 
of $A$ and $B$, denoted by $A \otimes B$, is the block matrix 
in $\mathbb{R}^{mp \times nq}$ defined by
$$
A \otimes B =
\begin{pmatrix}
a_{11}B & a_{12}B & \cdots & a_{1n}B \\
a_{21}B & a_{22}B & \cdots & a_{2n}B \\
\vdots  & \vdots  & \ddots & \vdots  \\
a_{m1}B & a_{m2}B & \cdots & a_{mn}B
\end{pmatrix}.
$$
\end{definition}

Let $A, C$ and $B, D$ be matrices of appropriate sizes. Then the Kronecker product satisfies the following properties:
\begin{itemize}
\item Mixed-product property: $
(A \otimes B)(C \otimes D) = (AC) \otimes (BD).$
\item Distributivity: $A \otimes (B + C) = A \otimes B + A \otimes C, \,\,\,\, $ 
$(A + C) \otimes B = A \otimes B + C \otimes B.
$
\item Scalar multiplication: $(\alpha A) \otimes B = A \otimes (\alpha B) = \alpha (A \otimes B).$
\item Transpose:
$(A \otimes B)^T = A^T \otimes B^T.$
\end{itemize}

\begin{proposition}
If $Ax = \lambda x$ and $By = \mu y$, then $
(A \otimes B)(x \otimes y) = (\lambda \mu)(x \otimes y).$
Thus $x \otimes y$ is an eigenvector of $A \otimes B$ corresponding 
to the eigenvalue $\lambda \mu$.
\end{proposition}

\begin{proposition}
For matrices $A, X, B$ of compatible dimensions, $
\mathrm{vec}(AXB) = (B^T \otimes A)\,\mathrm{vec}(X).$
\end{proposition}

The path graph $P_n$ is a graph with vertex set 
$\{1,2,\dots,n\}$ and edge set $\{(i,i+1) : 1 \leq i < n\}.$    
So, we consider path graph $P_n$ to be path graph consisting of $n$ vertices and $n-1$ edges.

\begin{definition}\cite{harary}
Given graphs $G_1=(V_1,E_1)$ and $G_2=(V_2,E_2)$, their Cartesian product $G = G_1 \square G_2$ is defined by $
V(G) = V_1 \times V_2,$ and
\[
((u_1,v_1),(u_2,v_2)) \in E(G) \iff
\begin{cases}
u_1 = u_2 \text{ and } (v_1,v_2) \in E_2, \\
\text{or} \\
v_1 = v_2 \text{ and } (u_1,u_2) \in E_1.
\end{cases}
\]    
\end{definition}

For example, consider $P_2$ and $P_3$. Then product $ (P_2 \Box P_3)$ is given by {\scriptsize
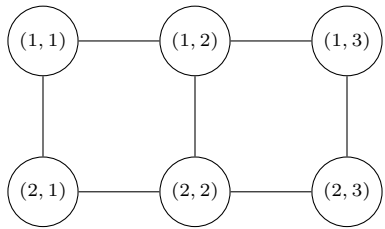
$$\begin{tikzpicture}[scale=1, every node/.style={circle, draw,minimum size = 1mm}]
  \node (v11) at (0,0) {$(1,1)$};
  \node (v12) at (2,0) {$(1,2)$};
  \node (v13) at (4,0) {$(1,3)$};
  
  \node (v21) at (0,-2) {$(2,1)$};
  \node (v22) at (2,-2) {$(2,2)$};
  \node (v23) at (4,-2) {$(2,3)$};
  
  \draw (v11) -- (v12);
  \draw (v12) -- (v13);
  
  \draw (v21) -- (v22);
  \draw (v22) -- (v23);
  
  \draw (v11) -- (v21);
  \draw (v12) -- (v22);
   \draw (v13) -- (v23);
\end{tikzpicture}$$}
Further, consider $P_3$ and $P_4$. Then $P_3 \square P_4$ is given by {\scriptsize
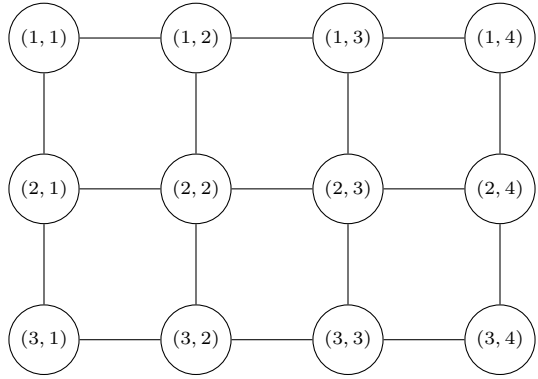
\[
\begin{tikzpicture}[scale=1, every node/.style={circle, draw, minimum size=1mm}]
  \node (v11) at (0,0) {$(1,1)$};
  \node (v12) at (2,0) {$(1,2)$};
  \node (v13) at (4,0) {$(1,3)$};
  \node (v14) at (6,0) {$(1,4)$};

  \node (v21) at (0,-2) {$(2,1)$};
  \node (v22) at (2,-2) {$(2,2)$};
  \node (v23) at (4,-2) {$(2,3)$};
  \node (v24) at (6,-2) {$(2,4)$};

  \node (v31) at (0,-4) {$(3,1)$};
  \node (v32) at (2,-4) {$(3,2)$};
  \node (v33) at (4,-4) {$(3,3)$};
  \node (v34) at (6,-4) {$(3,4)$};

  \foreach \i in {1,2,3} {
    \draw (v\i1) -- (v\i2);
    \draw (v\i2) -- (v\i3);
    \draw (v\i3) -- (v\i4);
  }

  \foreach \j in {1,2,3,4} {
    \draw (v1\j) -- (v2\j);
    \draw (v2\j) -- (v3\j);
  }

\end{tikzpicture}
\]}

It is known that $G_1 \square G_2$ is connected if and only if both the graphs $G_1$ and $G_2$ are connected~\cite{imrich2000product}.

The graph $P_2 \square P_m$ is a ladder graph with $2m$ vertices.
Vertices can be indexed as: $
V = \{(1,i), (2,i) : 1 \leq i \leq m\}.$
Edges consist of: horizontal edges along each path,
 vertical edges connecting the two paths. 

Now, the laplacian of $ L(P_2 \Box P_m)$ satisfies 
$$ L(P_2 \Box P_m) = L(P_2) \otimes I_m + I_2 \otimes L(P_m) = 
  \begin{pmatrix}
     L(P_m) + I_m & -I_m  \\
     -I_m & L(P_m) + I_m  \\
 \end{pmatrix} 
$$ with the adjacency matrix $A(P_2 \Box P_m) = A(P_2) \otimes I_m + I_2 \otimes A(P_m) $
and the Laplacian of the Cartesian product $P_n \square P_m$ satisfies
$$
L(P_n \square P_m)
=
L(P_n) \otimes I_m + I_n \otimes L(P_m) =L(P_n) \oplus_K L(P_m), 
$$ where $\oplus_K $ is the Kronecker sum. 

\begin{theorem}\cite{RMerris94} \label{eigproduthm}
Let $G$ and $H$ be graphs with Laplacian eigenpairs 
$(\lambda_i, x_i)$ and $(\mu_j, y_j)$. Then the Cartesian product 
$G \square H$ has eigenvalues $\lambda_i + \mu_j$ with corresponding 
eigenvectors $z = x_i \otimes y_j,$
i.e., $
z(v,w) = x_i(v)\, y_j(w).$
\end{theorem}

\section{Arithmetical Structures on $P_2 \square P_m$}
Consider $G = P_2 \square P_m$, the ladder graph with vertices $\{(1,i), (2,i): \quad i = 1, \dots, m\}.$ 

Define the Adjacency matrix $A(G)$  as
$$
A(G)_{(i,j),(p,q)} =
\begin{cases}
-1 & \text{if } (i,j) \sim (p,q),\\
0 & \text{otherwise}.
\end{cases}.
$$
We can split $r$ into row vectors $
r = \begin{pmatrix} r_1 \\ r_2 \end{pmatrix}, \quad r_1, r_2 \in \mathbb{Z}_{>0}^{m}.$

Recall that a pair $(\mathbf{d}, \mathbf{r})$ is an arithmetical structure on $G$ if:
\[
(\mathrm{diag}(\mathbf{d}) - A)\mathbf{r} = 0, \quad \mathbf{r} > 0, \quad \gcd(\mathbf{r})=1.
\]   
If we apply the definition of arithmetical structures on $P_2 \square P_m$,
we have \[
d_{1,i} \cdot r_{1,i} = \sum_{u \sim (1,i)} r_u
\]

That is, 
for the first row:
\[
d_{1,i} r_{1,i} =
\begin{cases}
r_{1,2} + r_{2,1}, & i=1,\\[2mm]
r_{1,i-1} + r_{1,i+1} + r_{2,i}, & 1 < i < m,\\[1mm]
r_{1,m-1} + r_{2,m}, & i=m.
\end{cases}
\]

For the second row:
\[
d_{2,i} r_{2,i} =
\begin{cases}
r_{2,2} + r_{1,1}, & i=1,\\[2mm]
r_{2,i-1} + r_{2,i+1} + r_{1,i}, & 1 < i < m,\\[1mm]
r_{2,m-1} + r_{1,m}, & i=m.
\end{cases}
\]
Now, consider the Laplacian of $P_2\square P_m$. By Theorem~\ref{eigproduthm}, $0 $ is an eigenvalue of $L(P_2 \square P_m)$ with corresponding eigenvector $\mathbf{1}$, the all-ones vector. Thus, we obtain an arithmetical structure given by ${\bf{r}} = {\bf 1}$ with ${\bf d} = \diag(d)$, where $d$ is the degree vector.

Note that $(P_2 \Box P_2)$ is  simply a cycle graph. Hence, the characterization of the arithmetical structures is reduced to that of  $C_4$.

However, arithmetical structures on 
$(P_2 \square P_m)$ for $m \geq 3$ exhibit a richer and more intricate behavior, requiring a more detailed analysis.

\begin{theorem}\label{ist}
Let $G = P_2 \square P_m$ be the ladder graph with vertices $(1,i), (2,i)$ for $1 \le i \le m$, and let $(\mathbf{d}, \mathbf{r}) \in \mathcal{A}(G)$ be an arithmetical structure on $G$. Suppose $\mathbf{r} \neq \mathbf{1}_{2m}$. Then one of the following holds:
\begin{enumerate}
    \item $d_{1,1} > 3$ or $d_{2,1} > 3$, and there exists a vertex $(i,j) \neq (1,1),(2,1)$ such that $d_{i,j} < 2$,
    
    \item $d_{1,1} < 3$ or $d_{2,1} < 3$, and there exists a vertex $(i,j) \neq (1,1),(2,1)$ such that $d_{i,j} > 2$,
    
    \item $d_{1,1} = 3$ and $d_{2,1} = 3$, and there exist vertices $(i,j), (k,l) \neq (1,1),(2,1)$ such that 
    $d_{i,j} < 2$ and $d_{k,l} > 2$.
\end{enumerate}
Moreover, if $d_{i,j} = 1$ for some vertex $(i,j)$, then
$d_{p,q} > 1 \quad \text{for all } (p,q) \in N_G(i,j)$, where $N_G((i,j))$ denotes the set of neighbors of vertex $(i,j)$ in $G$.
\end{theorem}

\begin{proof}
Let $G = P_2 \square P_m$ be the ladder graph with vertices $(1,i),(2,i)$ for $1 \le i \le m$, and let $(\mathbf{d}, \mathbf{r}) \in \mathcal{A}(G)$. Then
\[
(\operatorname{diag}(\mathbf{d}) - A)\mathbf{r} = 0,
\]
where $\mathbf{r}$ is a positive integer vector. Thus, for every vertex $(i,j)$,
\[
d_{i,j} \, r_{i,j} = \sum_{(p,q)\sim(i,j)} r_{p,q}.
\]
Equivalently,
\[
d_{i,j} = \frac{\sum_{(p,q)\sim(i,j)} r_{p,q}}{r_{i,j}}.
\]

If $\mathbf{r} = \mathbf{1}_{2m}$, then $d_{i,j} = \deg_G(i,j)$ for all vertices, and the structure is trivial. Hence assume $\mathbf{r} \neq \mathbf{1}_{2m}$.

Let $(i_0,j_0)$ be a vertex where $r_{i_0,j_0}$ is minimal, and $(k_0,l_0)$ be a vertex where $r_{k_0,l_0}$ is maximal.

Since $r_{i_0,j_0} \le r_{p,q}$ for all neighbors $(p,q)$ of $(i_0,j_0)$, we have
\[
d_{i_0,j_0} = \frac{\sum r_{p,q}}{r_{i_0,j_0}} \ge \deg_G(i_0,j_0),
\]
with strict inequality unless all neighboring values are equal.

Similarly, since $r_{k_0,l_0} \ge r_{p,q}$ for all neighbors $(p,q)$ of $(k_0,l_0)$,
\[
d_{k_0,l_0} = \frac{\sum r_{p,q}}{r_{k_0,l_0}} \le \deg_G(k_0,l_0),
\]
with strict inequality unless all neighboring values are equal.

Since $\mathbf{r}$ is not constant, at least one of these inequalities is strict. Therefore, by using Theorem~\ref{Theorem 2.6 from Corrales and Valencia} and Theorem~\ref{Theorem 3.2 from Corrales and Valencia} there exist vertices where $d_{i,j} > \deg_G(i,j)$ and $d_{i,j} < \deg_G(i,j)$.

In the ladder graph, corner vertices have degree $2$ and all other vertices have degree $3$. Hence deviations of $\mathbf{d}$ from the degree vector occur across the graph.

Focusing on the vertices $(1,1)$ and $(2,1)$, we obtain the three cases:
\begin{enumerate}
    \item If $d_{1,1} > 3$ or $d_{2,1} > 3$, then there exists $(i,j) \neq (1,1),(2,1)$ such that $d_{i,j} < 2$,
    
    \item If $d_{1,1} < 3$ or $d_{2,1} < 3$, then there exists $(i,j) \neq (1,1),(2,1)$ such that $d_{i,j} > 2$,
    
    \item If $d_{1,1} = 3$ and $d_{2,1} = 3$, then there exist $(i,j), (k,l) \neq (1,1),(2,1)$ such that 
    $d_{i,j} < 2$ and $d_{k,l} > 2$.
\end{enumerate}

Finally, suppose $d_{i,j} = 1$ for some vertex $(i,j)$. Then
\[
r_{i,j} = \sum_{(p,q)\sim(i,j)} r_{p,q}.
\]
Assume that there exists a neighbor $(p,q)$ such that $d_{p,q} = 1$. Then
\[
r_{p,q} = \sum_{(u,v)\sim(p,q)} r_{u,v}.
\]
Since $(i,j) \sim (p,q)$, we have
\[
r_{p,q} \ge r_{i,j} + (\text{other positive terms}) > r_{i,j}.
\]
Substituting into the first equation,
\[
r_{i,j} = r_{p,q} + (\text{other positive terms}) > r_{p,q},
\]
which is a contradiction. Hence $d_{p,q} > 1$ for all neighbors $(p,q)$ of $(i,j)$. 

The arguments for corner vertices $(1,1),(1,m),(2,1),(2,m)$ and interior vertices $(1,i),(2,i)$ for $1 < i < m$ 
follow similarly by considering the appropriate Laplacian equations for each type of vertex. That is,  consider a corner vertex, say $(1,1)$, with
\[
d_{1,1} r_{1,1} = r_{1,2} + r_{2,1}.
\]
If $d_{1,1} = 1$ and one of its neighbors also has $d=1$, say $(1,2)$ or $(2,1)$, then the corresponding Laplacian equation gives
\[
r_{1,1} = r_{1,2} + r_{2,1} = r_{1,1} + r_{x,y} > r_{1,1},
\]
where $(x,y)$ denotes the neighbor with $d_{x,y}=1$. This is a contradiction since all $r_i > 0$.
This completes the proof.
\end{proof}

\begin{proposition}
Let $G$ be a graph with $n$ vertices, and let $(\mathbf{d}, \mathbf{r})$ be an arithmetical structure on $G$ with respect to a fixed ordering of the vertex set. Let $P$ be any $n \times n$ permutation matrix corresponding to a relabeling of the vertices. Then $(\mathbf{d}', \mathbf{r}')$ defined by
\[
\mathbf{d}' = P\mathbf{d}, \quad \mathbf{r}' = P\mathbf{r}
\]
is also an arithmetical structure on $G$ (with respect to the new ordering). 

In particular, different orderings of the vertex set yield permutation-equivalent representations of the same arithmetical structure.
\end{proposition}

\begin{proof}
Let $A$ be the adjacency matrix of $G$ with respect to the original ordering. Then $(\mathbf{d}, \mathbf{r})$ satisfies
$(\operatorname{diag}(\mathbf{d}) - A)\mathbf{r} = 0.$
Let $P$ be a permutation matrix corresponding to a relabeling of the vertices, and define
\[
\mathbf{d}' = P\mathbf{d}, \quad \mathbf{r}' = P\mathbf{r}.
\]
Under this relabeling, the adjacency matrix transforms as
$A' = P A P^{T}.$
Similarly,
\[
\operatorname{diag}(\mathbf{d}') = P \operatorname{diag}(\mathbf{d}) P^{T}.
\]
Now consider $
(\operatorname{diag}(\mathbf{d}') - A')\mathbf{r}'.$
Substituting,
\[
(P \operatorname{diag}(\mathbf{d}) P^{T} - P A P^{T})(P\mathbf{r})
= P(\operatorname{diag}(\mathbf{d}) - A)\mathbf{r}.
\]
Since $(\operatorname{diag}(\mathbf{d}) - A)\mathbf{r} = 0$, it follows that $
(\operatorname{diag}(\mathbf{d}') - A')\mathbf{r}' = 0.$
Thus $(\mathbf{d}', \mathbf{r}')$ is an arithmetical structure on $G$ with respect to the new ordering.

The final statement follows immediately.
\end{proof}

\begin{proposition}
Let $P_2 \square P_m$ be the ladder graph with $2m$ vertices. 
Let $r = \mathbf{1}_{2m} = (1,1,\dots,1) \in \mathbb{Z}_{>0}^{2m}.$
\begin{enumerate}
    \item For vertices labeled as row wise as
$(1,1),(1, 2)\dots,(1,m),(2,1),(2,2)\dots,(2,m)$, we have  $
d = (\underbrace{2,3,\dots,3,2}_{\text{row 1}}, \underbrace{2,3,\dots,3,2}_{\text{row 2}}). $
\item For vertices labeled as column wise as $(1, 1), (2, 1), (1, 2), (2, 2), \ldots, (1, m), (2, m)$ we have $
d = (2, 2, 3, 3, \ldots, 3, 3, 2, 2),$ where the first two and
last two entries correspond to the corner vertices, and all interior vertices have degree $3.$ 
\end{enumerate}
Then $(d,r)$ is an arithmetical structure on $P_2 \square P_m$. Thus, $\mathbf{d}$ is uniquely determined as a function on the vertices of $G$, independent of the chosen ordering. Different vertex orderings produce vectors that are permutation-equivalent representations of the same degree function.
\end{proposition}

\begin{proof}
The proof follows as the consequence of Theorem~\ref{ist}.
\end{proof}

\begin{theorem}\label{adjmminustom}
Let $m \ge 2$. Then the adjacency matrix of the ladder graph$ (P_2 \square P_m)$ can be written in block form as
\[
\tilde{A} = A(P_2 \square P_m)=
\begin{pmatrix}
A(P_2 \square P_{m-1}) & B_{m-1} \\
B_{m-1}^T & A(P_2)
\end{pmatrix},
\]
where $A(P_2)=\begin{pmatrix}0&1\\1&0\end{pmatrix}$, and $B_{m-1}\in \mathbb{R}^{2(m-1)\times 2}$ is given by
$
B_{m-1}=
\begin{pmatrix}
0\\
\vdots\\
0\\
I_2
\end{pmatrix},$
i.e., all entries are zero except the last two rows, which form the identity matrix $(I_2)$.
\end{theorem}

\begin{proof}
Order the vertices of $P_2 \square P_m$ column-wise as
\[
(1,1),(2,1),(1,2),(2,2),\dots,(1,m),(2,m).
\]
With this ordering, the first $2(m-1)$ vertices correspond to $P_2 \square P_{m-1}$, and the last two vertices correspond to the new column $(1,m),(2,m)$. The adjacency relations are as follows:
\begin{itemize}
\item Edges within the first (m-1) columns give the block $A(P_2 \square P_{m-1}).$
\item The vertices $(1,m)$ and $(2,m) $ are adjacent, giving the block $A(P_2)$.
\item The only edges between the old and new vertices are $
(1,m-1)\sim (1,m), \quad (2,m-1)\sim (2,m).$
\end{itemize}
These two edges are represented by the matrix $(B_{m-1}),$ which has zeros everywhere except in the last two rows, corresponding to vertices $(1,m-1)$ and $(2,m-1)$, where it equals $I_2.$
Thus, the adjacency matrix takes the stated block form.
\end{proof}

\begin{observation}
Note here that if we know arithmetical structures on $(P_2 \square P_{m-1})$ and $P_2$  , then we calculate arithmetical structures on $(P_2 \square P_{m})$ with respect to some adjustment of matrix $B_{m-1}$. \end{observation}


The next theorem calculates arithmetical structures on $P_2 \square P_{m}$ from $P_2 \square P_{m-1}$.

\begin{theorem}
Let $ P_2 \square P_{m-1}$ be the ladder graph with column-wise ordering \\ $
(1,1), (2,1), (1,2), (2,2), \dots, (1,m-1), (2,m-1).$ Let $A(P_2 \square P_{m-1})$  be the adjacency matrix of ladder graph $(P_2 \square P_{m-1}).$ Let $(\bf{d,r})$  be an arithmetical structure on $P_2 \square P_{m-1}$. Assume that $r_{1, m-1}|1+r_{2, m-1}+ r_{1, m-2}$ and $r_{2, m-1}|1+r_{1, m-1}+ r_{2, m-2}$. Then  $({\bf\tilde{d}, \tilde{r}})$  given by 
$${\bf\tilde{d}} = \begin{pmatrix}
   d_{11} \\
   d_{21}\\
   d_{12}\\
   d_{22} \\
   \vdots\\
   d_{1,m-2}\\
   d_{2,m-2} \\
   1+r_{2, m-1}+ r_{1, m-2}/r_{1, m-1}\\
     1+r_{1, m-1}+ r_{2, m-2}/r_{2, m-1}\\
     1+r_{1, m-1}\\
     1+r_{2, m-1}
\end{pmatrix}, \,\,\,\,  
{\bf\tilde{r}} = \begin{pmatrix}
   r_{11} \\
   r_{21}\\
   r_{12}\\
   r_{22} \\
   \vdots\\
   r_{1,m-2}\\
   r_{2,m-2} \\
    r_{1, m-1}\\
      r_{2, m-1}\\
     1\\
     1
\end{pmatrix} $$
is an arithmetical structure on $ P_2 \square P_m$. 
\end{theorem}

\begin{proof}
 Consider the adjacency relation given in Theorem~\ref{adjmminustom}. Then we have $({\diag({\bf\tilde{d}}})- \tilde{A}){\bf\tilde{r}} = 0.$ 
\end{proof}

\begin{theorem}
Let $G = P_2 \square P_{m-1}$ be the ladder graph with column-wise ordering \\ $
(1,1), (2,1), (1,2), (2,2), \dots, (1,m-1), (2,m-1),$
and let $(\mathbf d, \mathbf r)$ be an arithmetical structure on $G$.
Fix $i \in \{1,2\}$ and $1 \le j \le m-1$. Assume that $r_{1,m-1}|r_{1,m-2} + r_{2,m-1} + r_{i,j},  \,\, r_{2,m-1}|r_{2,m-2} + r_{1,m-1} + 1$ $r_{i,j}|r_{1,m-1} + 1$. Define a vector $\tilde{\mathbf r}$ on $P_2 \square P_m$ by
\[
\tilde{\mathbf r} =
(r_{1,1}, r_{2,1}, r_{1,2}, r_{2,2}, \dots, r_{1,m-1}, r_{2,m-1}, r_{i,j}, 1).
\]
Define $\tilde{\mathbf d}$ componentwise as follows:
\begin{enumerate}
    \item For $1 \le k \le m-2$, $
    \tilde d_{1,k} = d_{1,k}, \quad \tilde d_{2,k} = d_{2,k}.$
    \item
    \[
    \tilde d_{1,m-1}
    = \frac{r_{1,m-2} + r_{2,m-1} + r_{i,j}}{r_{1,m-1}}, \,\,\,\,
    \tilde d_{2,m-1}
    = \frac{r_{2,m-2} + r_{1,m-1} + 1}{r_{2,m-1}}.
    \]
    \item
    \[
    \tilde d_{1,m}
    = \frac{r_{1,m-1} + 1}{r_{i,j}}, \,\,\,\,
    \tilde d_{2,m}
    = r_{2,m-1} + r_{i,j}.
    \]
\end{enumerate}
Then $(\tilde{\mathbf d}, \tilde{\mathbf r})$ satisfies
$(\operatorname{diag}(\tilde{\mathbf d}) - \tilde A)\tilde{\mathbf r} = 0,$
where $\tilde A$ is the adjacency matrix of $P_2 \square P_m$. Moreover, $\tilde{\mathbf r}$ is primitive. Hence, $(\tilde{\mathbf d}, \tilde{\mathbf r})$ is an arithmetical structure on $P_2 \square P_m$ if and only if all entries of $\tilde{\mathbf d}$ are positive integers.
\end{theorem}

\begin{proof}
Since $(\mathbf d, \mathbf r)$ is an arithmetical structure on $P_2 \square P_{m-1}$, it satisfies
\[
(\operatorname{diag}(\mathbf d) - A)\mathbf r = 0.
\]
Equivalently, for each vertex $(u,v)$ with $1 \le v \le m-1$, we have
\[
d_{u,v} r_{u,v} = \sum_{(p,q)\sim(u,v)} r_{p,q}.
\]
We extend $\mathbf r$ to $\tilde{\mathbf r}$ by appending the entries $r_{i,j}$ and $1$. Since $\tilde{\mathbf r}$ contains the entry $1$, it follows that $\gcd(\tilde{\mathbf r}) = 1$, and hence $\tilde{\mathbf r}$ is primitive.
Now, define $\tilde{\mathbf d}$ so that for every vertex of $P_2 \square P_m$, the relation
\[
\tilde d_{u,v} \tilde r_{u,v} = \sum_{(p,q)\sim(u,v)} \tilde r_{p,q}
\]
holds.
For vertices $(1,k)$ and $(2,k)$ with $1 \le k \le m-2$, the neighborhood structure is unchanged, so the original equations remain valid and $\tilde d_{u,k} = d_{u,k}$, see the adjacency matrix relation given in Theorem~\ref{adjmminustom}.

For the vertices $(1,m-1)$ and $(2,m-1)$, each gains one additional neighbor, namely $(1,m)$ and $(2,m)$, respectively. Substituting the corresponding values of $\tilde{\mathbf r}$ gives
\[
\tilde d_{1,m-1} r_{1,m-1}
= r_{1,m-2} + r_{2,m-1} + r_{i,j},
\]
\[
\tilde d_{2,m-1} r_{2,m-1}
= r_{2,m-2} + r_{1,m-1} + 1,
\]
which yield the stated formulas.
For the new vertices $(1,m)$ and $(2,m)$, their neighbors are $(1,m-1),(2,m)$ and $(2,m-1),(1,m)$, respectively. Using $\tilde r_{1,m} = r_{i,j}$ and $\tilde r_{2,m} = 1$, we obtain
\[
\tilde d_{1,m} r_{i,j} = r_{1,m-1} + 1,
\]
\[
\tilde d_{2,m} \cdot 1 = r_{2,m-1} + r_{i,j},
\]
which again gives the stated expressions.
Thus, all vertex equations are satisfied, and hence
\[
(\operatorname{diag}(\tilde{\mathbf d}) - \tilde A)\tilde{\mathbf r} = 0.
\]
Finally, $(\tilde{\mathbf d}, \tilde{\mathbf r})$ is an arithmetical structure if and only if all entries of $\tilde{\mathbf d}$ are positive integers.
\end{proof}

\begin{theorem}
Let $G = P_2 \square P_{m-1}$ be the ladder graph with column-wise ordering \\ $
(1,1), (2,1), (1,2), (2,2), \dots, (1,m-1), (2,m-1),$ and let $(\mathbf d, \mathbf r)$ be an arithmetical structure on $G$.
Fix $i \in \{1,2\}$ and $1 \le j \le m-1$. Define a vector $\tilde{\mathbf r}$ on $P_2 \square P_m$ by
\[
\tilde{\mathbf r} =
(r_{1,1}, r_{2,1}, r_{1,2}, r_{2,2}, \dots, r_{1,m-1}, r_{2,m-1}, 1, r_{i,j}).
\]
Define $\tilde{\mathbf d}$ componentwise as follows:
\begin{enumerate}
    \item For $1 \le k \le m-2$, $
    \tilde d_{1,k} = d_{1,k}, \quad \tilde d_{2,k} = d_{2,k}.$
    \item
    \[
    \tilde d_{1,m-1}
    = \frac{r_{1,m-2} + r_{2,m-1} + 1}{r_{1,m-1}}, \,\,\,\,
    \tilde d_{2,m-1}
    = \frac{r_{2,m-2} + r_{1,m-1} + r_{i,j}}{r_{2,m-1}}.
    \]
    \item
    \[
    \tilde d_{1,m}
    = r_{1,m-1} + r_{i,j}, \,\,\,\,
    \tilde d_{2,m}
    = \frac{r_{2,m-1} + 1}{r_{i,j}}.
    \]
\end{enumerate}
Then $(\tilde{\mathbf d}, \tilde{\mathbf r})$ satisfies
$
(\operatorname{diag}(\tilde{\mathbf d}) - \tilde A)\tilde{\mathbf r} = 0,$
where $\tilde A$ is the adjacency matrix of $P_2 \square P_m$. Moreover, $\tilde{\mathbf r}$ is primitive. Hence, $(\tilde{\mathbf d}, \tilde{\mathbf r})$ is an arithmetical structure on $P_2 \square P_m$ if and only if all entries of $\tilde{\mathbf d}$ are positive integers.
\end{theorem}



\begin{exam}
Consider the graph $P_2 \square P_m$ with $m=4$. Then by Theorem~\ref{adjmminustom} we have
$$
A(P_2 \square P_4)=
\begin{pmatrix}
A(P_2 \square P_3) & B_3 \\
B_3^T & A(P_2)
\end{pmatrix},
\quad
B_3=
\begin{pmatrix}
0&0\\
0&0\\
0&0\\
0&0\\
1&0\\
0&1
\end{pmatrix},   A(P_2) =\begin{pmatrix}0&1\\1&0\end{pmatrix}.
$$
Now, 
$${\bf\tilde{d}} = \begin{pmatrix}
   d_{11} \\
   d_{21}\\
   d_{12}\\
   d_{22} \\
   1+r_{2,3}+ r_{1,2}/r_{1, 3}\\
     1+r_{1, 3}+ r_{2,2}/r_{2, 3}\\
     1+r_{1, 3}\\
     1+r_{2, 3}
\end{pmatrix} \text{   and    } {\bf\tilde{r}} = \begin{pmatrix}
   r_{11} \\
   r_{21}\\
   r_{12}\\
   r_{22} \\
    r_{13}\\
      r_{23}\\
     1\\
     1
\end{pmatrix} $$ is an arithmetical structures on $P_2 \square P_4.$
\end{exam}

\begin{theorem}[Recursive construction of arithmetical structures on ladder graphs]
Let $P_2 \square P_m$ be the ladder graph with $m \ge 2$, and let 
$(\mathbf{r}_{m-1}, \mathbf{d}_{m-1})$ be an arithmetical structure on $P_2 \square P_{m-1}$, 
with $\mathbf{r}_{m-1} = (r_1, \dots, r_{2m-2})^\top$.  

Then all arithmetical structures $(\mathbf{r}_m, \mathbf{d}_m)$ on $P_2 \square P_m$ 
can be obtained recursively as follows:
\begin{enumerate}
    \item Choose positive integers $y_1, y_2 \in \mathbb{Z}_{>0}$ satisfying the divisibility conditions
    \[
        y_1 \mid r_{2m-3} + y_2, \quad y_2 \mid r_{2m-2} + y_1.
    \]
    \item Define
    \[
        \mathbf{r}_m = (r_1, \dots, r_{2m-2}, y_1, y_2)^\top.
    \]
    \item Define the corresponding degrees by
    \[
        d_i = \frac{\sum_{j \sim i} r_j}{r_i}, \quad i = 1, \dots, 2m.
    \]
\end{enumerate}

Then $(\mathbf{r}_m, \mathbf{d}_m)$ is an arithmetical structure on $P_2 \square P_m$. 
Conversely, every arithmetical structure on $P_2 \square P_m$ arises in this way from 
an arithmetical structure on $P_2 \square P_{m-1}$.
\end{theorem}

\begin{proof}
The adjacency matrix of $P_2 \square P_m$ can be written in block form as
\[
\tilde{A} = 
\begin{pmatrix}
A(P_2 \square P_{m-1}) & B_{m-1} \\
B_{m-1}^\top & A(P_2)
\end{pmatrix},
\]
where $B_{m-1}$ has zeros everywhere except the last two rows, which form $I_2$.

Let $\mathbf{r}_m = \begin{pmatrix} \mathbf{r}_{m-1} \\ y_1 \\ y_2 \end{pmatrix}$. 
By definition, $(\mathbf{r}_m, \mathbf{d}_m)$ is an arithmetical structure if
\[
(\mathrm{diag}(\mathbf{d}_m) - \tilde{A}) \mathbf{r}_m = 0.
\]

Writing in block form gives
\[
\begin{pmatrix}
\mathrm{diag}(\mathbf{d}_{m-1}) - A(P_2 \square P_{m-1}) & -B_{m-1} \\
-B_{m-1}^\top & \mathrm{diag}(\mathbf{d}_2) - A(P_2)
\end{pmatrix}
\begin{pmatrix} \mathbf{r}_{m-1} \\ y_1 \\ y_2 \end{pmatrix} = 0.
\]

The bottom block gives the system
\[
(\mathrm{diag}(\mathbf{d}_2) - A(P_2)) \begin{pmatrix} y_1 \\ y_2 \end{pmatrix} 
- B_{m-1}^\top \mathbf{r}_{m-1} = 0,
\]
which reduces to the divisibility conditions for $y_1, y_2$:
\[
y_1 \mid r_{2m-3} + y_2, \quad y_2 \mid r_{2m-2} + y_1.
\]

Once $y_1, y_2$ are chosen, define degrees by
\[
d_i = \frac{\sum_{j \sim i} r_j}{r_i}, \quad i = 1, \dots, 2m.
\]
This ensures $(\mathbf{r}_m, \mathbf{d}_m)$ satisfies 
$(\mathrm{diag}(\mathbf{d}_m) - \tilde{A}) \mathbf{r}_m = 0$ with all positive integers.

Conversely, any arithmetical structure on $P_2 \square P_m$ restricts to 
an arithmetical structure on $P_2 \square P_{m-1}$ and the last two entries satisfy the same divisibility conditions. Hence all arithmetical structures arise recursively.
\end{proof}

\subsection{Arithmetical structures of $P_2 \Box P_m$ from $\mathcal{A}(P_m)$}
In this section, we define arithmetical structures on $P_2 \square P_m$ from $\mathcal{A}(P_m)$

\begin{theorem}\label{stac}[Stacked Theorem]
Let $m \ge 2$, and let $(d_1, r_1)$  be an arithmetical structure on $P_m$, i.e.,
$$
\operatorname{diag}(d_1)\, r_1 = A(P_m)\, r_1,
\quad r_1 \in \mathbb{Z}_{>0}^m,\ d_1 \in \mathbb{Z}_{>0}^m.
$$
Let $P_2 \square P_m$ be a ladder graph with vertex ordering $(1,1),\ldots,(1,m),(2,1), \ldots,(2,m).$
Then arithmetical structures on $P_2 \square P_m$ can be constructed as follows: 
\begin{enumerate}
    \item   
    Define
    \[
    {\bf r} =
    \begin{pmatrix}
    r_1 \\
    r_1
    \end{pmatrix}, \quad
    {\bf d} =
    \begin{pmatrix}
    d_1 + \mathbf{1} \\
    d_1 + \mathbf{1}
    \end{pmatrix},
    \]
    where $\mathbf{1} = (1,\dots,1) \in \mathbb{Z}^m$.  
    Then $({\bf d},{\bf r})$ is an arithmetical structure on $P_2 \square P_m$.

    \item 
    Let $k \in \mathbb{Z}_{\ge 0}^m$, and define $
    D_1 = \operatorname{diag}(d_1 + k).$
    Set $
    r_2 := (D_1 - A(P_m))\, r_1.$
    If $r_2 \in \mathbb{Z}_{>0}^m$, and $
r_2 \mid \big(A(P_m)\, r_2 + r_1\big)
\quad \text{(componentwise)}.$ Define $d_2 \in \mathbb{Z}_{>0}^m$ by
\[
(d_2)_i = \frac{(A(P_m)\, r_2 + r_1)_i}{(r_2)_i}, 
\quad 1 \le i \le m.
\]
That is, there exists $d_2 \in \mathbb{Z}_{>0}^m$ such that $(D_2 - A(P_m))\, r_2 = r_1,$
where $D_2 = \operatorname{diag}(d_2)$, then
\[
r =
    \begin{pmatrix}
    r_1 \\
    r_2
    \end{pmatrix}, \quad
    d =
    \begin{pmatrix}
    d_1 + k \\
    d_2
    \end{pmatrix}
    \]
    is an arithmetical structure on $P_2 \square P_m$.
\end{enumerate}
\end{theorem}

\begin{proof}
Let $(d_1, r_1)$ be an arithmetical structure on $P_m$, so that $\operatorname{diag}(d_1)\, r_1 = A(P_m)\, r_1.$ 
The adjacency matrix of $P_2 \square P_m$ is
\[
A(P_2 \square P_m) =
\begin{pmatrix}
A(P_m) & I_m \\
I_m & A(P_m)
\end{pmatrix}.
\]
Let
\[
r =
\begin{pmatrix}
r_1 \\
r_2
\end{pmatrix}, \quad
d =
\begin{pmatrix}
d_1' \\
d_2
\end{pmatrix},
\]
where $d_1' = d_1 + k$, and let $D_i = \operatorname{diag}(d_i), i=1, 2$.
Then the condition $(D - A)r = 0$ becomes
\[
\begin{pmatrix}
D_1 - A(P_m) & -I_m \\
- I_m & D_2 - A(P_m)
\end{pmatrix}
\begin{pmatrix}
r_1 \\
r_2
\end{pmatrix}
= 0,
\]
which is equivalent to the system
\[
(D_1 - A(P_m))r_1 = r_2, \qquad
(D_2 - A(P_m))r_2 = r_1.
\]
For the diagonal entries, take $k = \mathbf{1}$, so $D_1 = \operatorname{diag}(d_1 + \mathbf{1})$. Then
\[
D_1 r_1 = \operatorname{diag}(d_1)r_1 + r_1 = A(P_m)r_1 + r_1,
\]
and hence $(D_1 - A(P_m))r_1 = r_1.$
Thus $r_2 = r_1$. Taking $d_2 = d_1 + \mathbf{1}$, the second equation is also satisfied, and therefore $(d,r)$ is an arithmetical structure.

Let $k \in \mathbb{Z}_{\ge 0}^m$ and define $D_1 = \operatorname{diag}(d_1 + k)$. Set $
r_2 = (D_1 - A(P_m))r_1.$
If $r_2 \in \mathbb{Z}_{>0}^m$, then the first equation holds.
To satisfy the second equation, we require
\[
(D_2 - A(P_m))r_2 = r_1,
\]
which is equivalent to $D_2 r_2 = A(P_m)r_2 + r_1.$
Since $r_2 > 0$, this determines $d_2$ componentwise by
\[
(d_2)_i = \frac{(A(P_m)r_2 + r_1)_i}{(r_2)_i}.
\]
If these values are positive integers, then $d_2 \in \mathbb{Z}_{>0}^m$, and both equations are satisfied. Hence $(D - A)r = 0$, and $(d,r)$ is an arithmetical structure on $P_2 \square P_m$.
\end{proof}

\begin{remark}
Although $k \in \mathbb{Z}_{\ge 0}^m$ is arbitrary in the construction, the
positivity and divisibility conditions impose strong restrictions. Since
$P_2 \square P_m$ admits only finitely many arithmetical structures, it
follows that only finitely many choices of $k$ yield valid constructions.

In particular, the entries of $k$ are effectively bounded by the conditions
\[
r_2 = (D_1 - A(P_m))r_1 > 0
\quad \text{and} \quad
r_2 \mid (A(P_m)r_2 + r_1),
\]
and hence the construction produces only finitely many distinct arithmetical
structures.
\end{remark}

\begin{remark}
Let $P_2 \square P_m$ be the Cartesian product graph. The set of arithmetical structures $(\mathbf{d}, \mathbf{r})$ on $P_2 \square P_m$ satisfying $
r(1,i) = r(2,i) \quad \text{for all } i = 1, \dots, m,$
is in bijection with the set of arithmetical structures on $P_m$. 
\end{remark}

\begin{remark}
The above lemma describes only a subclass of arithmetical structures on $P_2 \square P_m$. In general, there exist non-symmetric arithmetical structures that do not satisfy $r(1,i)=r(2,i)$.
\end{remark}

\begin{observation}
 We have $  C_{m-1} \leq \#$ arith $(P_2 \square P_m) \leq C_{m-1}^2$.  
\end{observation}

\begin{exam}
Consider the path $P_4$ with 4 vertices. Its primitive arithmetical structures $r_1 \in \mathbb{Z}_{>0}^4$ with corresponding $d_1 \in \mathbb{Z}_{>0}^4$ are:
\[
\begin{array}{c|c}
r_1 & d_1 \\
\hline
(1,2,1,1) & (2,1,3,1) \\
(1,1,2,1) & (1,3,1,2) \\
(1,1,1,1) & (1,2,2,1)
\end{array}
\]
Using the construction from the theorem, we define
\[
r = \begin{pmatrix} r_1 \\ r_1 \end{pmatrix}, \quad
d = \begin{pmatrix} d_1 + 1 \\ d_1 + 1 \end{pmatrix}.
\]
Then the corresponding arithmetical structures $(d,r)$ on $P_2 \square P_4$ are:
\[
\begin{array}{c|c}
r & d \\
\hline
(1,2,1,1,1,2,1,1) & (3,2,4,2,3,2,4,2) \\
(1,1,2,1,1,1,2,1) & (2,4,2,3,2,4,2,3) \\
(1,1,1,1,1,1,1,1) & (2,3,3,2,2,3,3,2)
\end{array}
\]
These examples illustrate how the arithmetical structures on $P_2 \square P_4$ can be constructed from the arithmetical structures of $P_4$ using the theorem.
\end{exam}

\begin{theorem}
Let $G = P_2 \square P_m$ be the ladder graph on $2m$ vertices, and let 
$\mathcal{A}(G)$ denote the set of primitive arithmetical structures on a graph $G$. Define $
N(m) = |\mathcal{A}(G)|.$
Then $
C_{m-1} \le N(m),$
where $C_{m-1}$ is the $(m-1)$-th Catalan number. 
\end{theorem}

\begin{proof}
Let $({\bf d},{\bf r})$ be a primitive arithmetical structure on $G = P_2 \square P_m$. 
Write
\[
{\bf r} =
\begin{pmatrix}
r_1 \\
r_2
\end{pmatrix}, 
\quad r_1, r_2 \in \mathbb{Z}_{>0}^m.
\]

The arithmetical structure condition is $
(\operatorname{diag}(d) - A(G))r = 0,$
which is equivalent to the entrywise condition
$
d_i r_i = \sum_{j \sim i} r_j.$

For each $1 \le i \le m$, using the ladder structure, we obtain:
\begin{equation}
d_{1,i} r_{1,i} = r_{1,i-1} + r_{1,i+1} + r_{2,i}, \tag{1}
\end{equation}
\begin{equation}
d_{2,i} r_{2,i} = r_{2,i-1} + r_{2,i+1} + r_{1,i}, \tag{2}
\end{equation}
with appropriate boundary conditions at $i=1,m$. Consider symmetric vectors $r_1 = r_2 = x$. Then (1)–(2) reduce to
\[
d_i x_i = x_{i-1} + x_{i+1} + x_i,
\]
which is equivalent to $
x_i \mid (x_{i-1} + x_{i+1}).$
Thus $x$ is an arithmetical structure on $P_m$. Hence
\[
N(m) \ge |\mathcal{A}(P_m)| = C_{m-1}.
\]


\end{proof}

Now, we derive the arithmetical structures on Cartesian product via $M$-matrices.

\begin{theorem}[Arithmetical structures on Cartesian product via $M$-matrices).]
Let $G$ and $H$ be finite connected graphs. Suppose $(d^G, x)$ and $(d^H, y)$ are arithmetical structures on $G$ and $H$, respectively. Define
\[
M_G = \operatorname{diag}(d^G) - A_G, \quad
M_H = \operatorname{diag}(d^H) - A_H,
\]
where $A_G$ and $A_H$ are the adjacency matrices of $G$ and $H$.
Then there exists an arithmetical structure $(({\bf d}, {\bf r}))$ on the Cartesian product $G \square H$ given by
\[
{\bf r}=x \otimes y, {\text{    that is ,    } }r_{(i,j)} = x_i y_j, \quad
d_{(i,j)} = d^G_i + d^H_j.
\]
Equivalently, if
$
M_{G \square H} := \operatorname{diag}({\bf d}) - A_{G \square H},
$
then
\[
M_{G \square H} = M_G \otimes I + I \otimes M_H,
\text{    and    }
M_{G \square H}(x \otimes y) = 0.
\]
\end{theorem}

\begin{proof}
Recall that the adjacency matrix of the Cartesian product satisfies
\[
A_{G \square H} = A_G \otimes I + I \otimes A_H.
\]
Let ${\bf r} = x \otimes y$. Using properties of the Kronecker product, we compute:
\[
(A_G \otimes I)(x \otimes y) = (A_G x) \otimes y,
\]
\[
(I \otimes A_H)(x \otimes y) = x \otimes (A_H y).
\]
Since $(d^G, x)$ and $(d^H, y)$ are arithmetical structures, we have
\[
A_G x = d^G \circ x, \quad A_H y = d^H \circ y,
\]
where $\circ$ denotes componentwise multiplication.
Thus,
\[
A_{G \square H}(x \otimes y) = (A_G x) \otimes y +x \otimes (A_H y)
= (d^G \circ x) \otimes y + x \otimes (d^H \circ y).
\]
Now, define $d$ on $G \square H$ by $
d_{(i,j)} = d^G_i + d^H_j.$ Then
\begin{align*}
   ( \diag(d)r)_{ij} &= d_{ij}r_{ij} = (d_{G}(i) +d_{H}(j))x_i y_j\\
   & = (d_{G}(i) x_i) y_j+(d_{H}(j))x_i y_j\\
   &= (d_{G}(i) x_i) y_j+x_i(d_{H}(j) y_j\\
   &\implies (d_{G} \circ x)_i = d_{G}(i)x_i, \,\, (d_{H} \circ y)_j = d_{H}(j)y_j \\
   &= (diag(d) r)_{ij} = (d_G \circ x)_{i}y_j
+ (x_i(d_{H} \circ y))_j
\end{align*}
So, we have $$\operatorname{diag}({\bf d})(x \otimes y)= 
 (d^G \circ x)\otimes y + x \otimes (d^H \circ y).$$
Now, $A_{G \square H}(x \otimes y) =\operatorname{diag}({\bf d})(x \otimes y).$  Given that $M_{G \square H} := \operatorname{diag}({\bf d}) - A_{G \square H}$.
Therefore, we have
\[
M_{G \square H}(x \otimes y)
= \operatorname{diag}({\bf d})(x \otimes y) - A_{G \square H}(x \otimes y) = 0.
\]
Since $(x > 0), \,\,\gcd(x) =1$ and $(y > 0), \,\, \gcd(y)=1$, it follows that $(r = x \otimes y > 0)$ and $\gcd(r)=1$. Also, standard properties of Kronecker sums imply that
\[
\operatorname{rank}(M_{G \square H}) = |V(G)||V(H)| - 1,
\]
so, by Theorem~\ref{Theorem 3.2 from Corrales and Valencia}, $(M_{G \square H})$ is a singular irreducible $M$-matrix.
Hence, by Theorem~\ref{Theorem 3.2 from Corrales and Valencia}, $({\bf d}, {\bf r})$ is an arithmetical structure on $(G \square H)$.
\end{proof}

Now, using the above result, we can give another proof of Theorem~\ref{stac}.

\begin{theorem}
Let $P_2$ and $P_m$ be path graphs with $2$ and $m$ vertices, respectively. 
Suppose $(d_1, r_1)$ is an arithmetical structure on $P_m$. Then an arithmetical structure 
$({\bf d}, {\bf r})$ on the ladder graph $P_2 \square P_m$ is given by
\[
{\bf r} = \begin{pmatrix} r_1 \\ r_1 \end{pmatrix} \in \mathbb{Z}_{>0}^{2m}, 
\quad
{\bf d} = \begin{pmatrix} d_1 + \mathbf{1} \\ d_1 + \mathbf{1} \end{pmatrix} \in \mathbb{Z}_{>0}^{2m},
\]
where $\mathbf{1} = (1,1,\dots,1) \in \mathbb{Z}^m$.
Equivalently, if $M_{P_2 \square P_m} := \mathrm{diag}(d) - A(P_2 \square P_m),$
then $M_{P_2 \square P_m} \, r = 0.$
\end{theorem}

\begin{proof}
Consider the arithmetical structure on $P_2$, which is unique:
\[
x = (1,1), \quad d_{P_2} = (1,1).
\]
Let $(d_1, r_1)$ be any arithmetical structure on $P_m$, with adjacency $A(P_m)$ and 
\[
M_{P_m} := \mathrm{diag}(d_1) - A(P_m), \quad M_{P_m} r_1 = 0.
\]

Using the Cartesian product construction for graphs, define
\[
r = x \otimes r_1 = \begin{pmatrix} r_1 \\ r_1 \end{pmatrix}, 
\quad 
d(i,j) = d_{P_2,i} + d_{1,j} = 1 + d_{1,j}.
\]

The adjacency matrix of $P_2 \square P_m$ satisfies
\[
A(P_2 \square P_m) = A(P_2) \otimes I_m + I_2 \otimes A(P_m),
\]
where $I_m$ is the $m \times m$ identity matrix. Now, compute
\[
A(P_2 \square P_m) \, r = (A(P_2) \otimes I_m)(x \otimes r_1) + (I_2 \otimes A(P_m))(x \otimes r_1).
\]
Using properties of Kronecker products we have
\[
(A(P_2) \otimes I_m)(x \otimes r_1) = (A(P_2)x) \otimes r_1 = d_{P_2} \circ x \otimes r_1 = \mathbf{1} \otimes r_1 = r_1,
\]
\[
(I_2 \otimes A(P_m))(x \otimes r_1) = x \otimes (A(P_m) r_1) = x \otimes (d_1 \circ r_1) = r_1.
\]
Hence, $A(P_2 \square P_m) \, {\bf r} = \mathrm{diag}({\bf d}) \, {\bf r},$
so
\[
M_{P_2 \square P_m}  {\bf r} = (\mathrm{diag}(d) - A(P_2 \square P_m))  {\bf r} = 0.
\]
Since $r_1 > 0$, it follows that ${\bf r} > 0$, and ${\bf d} > 0$. Therefore, $({\bf d}, {\bf r})$ is an arithmetical structure on $P_2 \square P_m$, as claimed.
\end{proof}

\begin{theorem}
Let $G = P_2 \square P_m$. Let $\mathbf{r}_1 \in \mathbb{Z}_{>0}^m$, and let $D_1, D_2$ be diagonal matrices with positive integer entries. Define ${\bf d} \in \mathbb{Z}_{>0}^{2m}$ by
\[
\operatorname{diag}({\bf d}) =
\begin{pmatrix}
D_1 & 0 \\
0 & D_2
\end{pmatrix}.
\]
For an integer $k \ge 0$, define
\[
\mathbf{r} =
\begin{pmatrix}
\mathbf{r}_1 \\
\mathbf{r}_1 + k\mathbf{1}
\end{pmatrix}
\in \mathbb{Z}_{>0}^{2m},
\]
where $\mathbf{1} = (1, \dots, 1)$. Suppose that $r \in \mathbb{Z}_{>0}^{2m}$  and $\gcd(\mathbf{r})=1$.
Then $(d,\mathbf{r})$ defines an arithmetical structure on $G$ if and only if
$$(D_1 - A(P_m) - I)\mathbf{r}_1 = k\mathbf{1}, \,\,\, (D_2 - A(P_m))(\mathbf{r}_1 + k\mathbf{1}) = \mathbf{r}_1.
$$
In particular, if such a $k$ exists, then it is uniquely determined.
\end{theorem}

\begin{proof}
Let
\[
\mathbf{r} =
\begin{pmatrix}
\mathbf{r}_1 \\
\mathbf{r}_1 + k\mathbf{1}
\end{pmatrix}.
\]
Recall that
\[
M(G) = \operatorname{diag}(\mathbf{d}) - A(G) =
\begin{pmatrix}
D_1 - A(P_m) & -I_m \\
-I_m & D_2 - A(P_m)
\end{pmatrix}.
\]
Now,  the first block of $M(G)\mathbf{r}$ gives
\[
(D_1 - A(P_m))\mathbf{r}_1 - (\mathbf{r}_1 + k\mathbf{1})
= (D_1 - A(P_m) - I)\mathbf{r}_1 - k\mathbf{1}.
\]
Thus, the first block vanishes if and only if 
\begin{equation}\label{eq1}
    (D_1 - A(P_m) - I)\mathbf{r}_1 = k\mathbf{1}.
\end{equation}
Similarly, the second block of  $M(G)\mathbf{r}$ gives
\[
(D_2 - A(P_m))(\mathbf{r}_1 + k\mathbf{1}) - \mathbf{r}_1.
\]
Thus, the second block vanishes if and only if
\begin{equation}\label{eq2}
 (D_2 - A(P_m))(\mathbf{r}_1 + k\mathbf{1}) = \mathbf{r}_1. 
\end{equation}
Therefore, $M(G)\mathbf{r} = \mathbf{0}$ if and only if both conditions (\ref{eq1}) and (\ref{eq2}) hold.
Finally, from condition (\ref{eq1}), the vector $(D_1 - A(P_m) - I)\mathbf{r}_1$ is constant, and its common value is $k$. Hence $k$ is uniquely determined.
\end{proof}

\begin{theorem} \label{nonsym}[Conditional Construction of Non-Symmetric Arithmetical Structures on $P_2 \square P_m$]
Let $G = P_2 \square P_m$ be the ladder graph with $m \ge 2$, and impose boundary conditions $
a_0 = b_0 = a_{m+1} = b_{m+1} = 0.$ We label the vertices of $P_2 \square P_m$ as $
(1,1),\dots,(1,m),(2,1),\dots,(2,m),$ 
and define $${\bf r} = (a_1,\dots,a_m, b_1,\dots,b_m).$$
Let $(a_i)_{i=1}^m$ and $(b_i)_{i=1}^m$ be sequences of positive integers with $a_i= r_{1,i}, \, b_{i} = r_{2,i}$ such that for all $i = 1,\dots,m$,
\[
a_i \mid (a_{i-1} + a_{i+1} + b_i), \qquad
b_i \mid (b_{i-1} + b_{i+1} + a_i).
\]
Define
\[
d_{1,i} = \frac{a_{i-1} + a_{i+1} + b_i}{a_i}, \qquad
d_{2,i} = \frac{b_{i-1} + b_{i+1} + a_i}{b_i}.
\]
Let
\[
{\bf r} = (a_1,\dots,a_m, b_1,\dots,b_m), \qquad
{\bf d} = (d_{1,1},\dots,d_{1,m}, d_{2,1},\dots,d_{2,m}).
\] If ${\bf r}$ is primitive, 
then $({\bf d}, {\bf r})$ is an arithmetical structure on $G$, i.e.,
$\operatorname{diag}(d)r = A(G)r.$
Moreover, if $(a_1,\dots,a_m) \neq (b_1,\dots,b_m)$, then the structure is non-symmetric.
\end{theorem}

\begin{proof}
By assumption,
\[
a_i \mid (a_{i-1} + a_{i+1} + b_i), \qquad
b_i \mid (b_{i-1} + b_{i+1} + a_i),
\]
so there exist integers $d_{1,i}, d_{2,i} \in \mathbb{Z}_{>0}$ such that
\[
d_{1,i} a_i = a_{i-1} + a_{i+1} + b_i, \qquad
d_{2,i} b_i = b_{i-1} + b_{i+1} + a_i.
\]

In the ladder graph $P_2 \square P_m$, each vertex $(1,i)$ is adjacent to $(1,i-1)$, $(1,i+1)$, and $(2,i)$, while each vertex $(2,i)$ is adjacent to $(2,i-1)$, $(2,i+1)$, and $(1,i)$. Using the boundary conditions $a_0 = b_0 = a_{m+1} = b_{m+1} = 0$, these relations hold uniformly for all $i = 1,\dots,m$. Therefore, we have
\[
(A(G)r)_{1,i} = a_{i-1} + a_{i+1} + b_i, \qquad
(A(G)r)_{2,i} = b_{i-1} + b_{i+1} + a_i.
\]

Comparing with the above expressions, we obtain
\[
d_{1,i} a_i = (A(G)r)_{1,i}, \qquad
d_{2,i} b_i = (A(G)r)_{2,i},
\]
and hence
\[
\operatorname{diag}({\bf d}){\bf r} = A(G){\bf r}.
\]
Thus $({\bf d}, {\bf r})$ is an arithmetical structure on $G$. Finally, if $(a_1,\dots,a_m) \neq (b_1,\dots,b_m)$, then the two layers differ, and the structure is non-symmetric.
\end{proof}


Now, we illustrate the above theorem with the following examples.

\begin{exam}
Consider the ladder graph $P_2 \square P_3$ with boundary conditions:
$
a_0 = b_0 = a_4 = b_4 = 0$.
Choose a non-symmetric first column:
\[
(a_1, b_1) = (2,1).
\]
Compute the second column using arithmetical structure equations:
\[
d_{2,1} = \frac{b_0 + b_2 + a_1}{b_1} = \frac{0 + b_2 + 2}{1} = b_2 + 2
\]
Choose $b_2 = 1 \implies d_{2,1} = 3$

\[
d_{1,1} = \frac{a_0 + a_2 + b_1}{a_1} = \frac{0 + a_2 + 1}{2} \implies a_2 = 1 \implies d_{1,1} = 1.
\]

Now, we have the third column
\[
d_{2,2} = \frac{b_1 + b_3 + a_2}{b_2} = \frac{1 + b_3 + 1}{1} = b_3 + 2.
\]
Choose $b_3 = 1 \implies d_{2,2} = 3$.

\[
d_{1,2} = \frac{a_1 + a_3 + b_2}{a_2} = \frac{2 + a_3 + 1}{1} = a_3 + 3
\]
Choose $a_3 = 1 \implies d_{1,2} = 4.$
Now, the last column degree becomes
\[
d_{1,3} = \frac{a_2 + a_4 + b_3}{a_3} = \frac{1 + 0 + 1}{1} = 2, \quad
d_{2,3} = \frac{b_2 + b_4 + a_3}{b_3} = \frac{1 + 0 + 1}{1} = 2
\]

Thus we have 
\[
\begin{array}{c|c|c|c|c}
i & a_i & b_i & d_{1,i} & d_{2,i} \\
\hline
1 & 2 & 1 & 1 & 3 \\
2 & 1 & 1 & 4 & 3 \\
3 & 1 & 1 & 2 & 2
\end{array}.
\]

Note that column 1 is non-symmetric ($a_1 \neq b_1$). All degrees $d_{1,i}, d_{2,i}$ are positive integers. Further, equations
\[
d_{1,i} a_i = a_{i-1} + a_{i+1} + b_i, \quad
d_{2,i} b_i = b_{i-1} + b_{i+1} + a_i
\]
are satisfied for all $i$. By Theorem~\ref{nonsym}, we have ${\bf r}= (2, 1, 1, 1, 1, 1)$ with $\gcd({\bf r}) = 1$ and ${\bf d} = (1, 4, 2, 3, 3, 2).$
Thus, $({\bf d}, {\bf r})$ satisfies $\operatorname{diag}({\bf d}){\bf r} = A(G){\bf r}$, where adjacency matrix $A(G)$ is row-wise ordering,  and hence is a valid non-symmetric arithmetical structure on $P_2 \square P_3$.
\end{exam}

\begin{exam}
Consider the ladder graph $P_2 \square P_4$ with boundary conditions
$a_0 = b_0 = a_5 = b_5 = 0.$
Let the vertex labels be given by
\[
(a_1,b_1) = (2,1), \quad (a_2,b_2) = (1,1), \quad (a_3,b_3) = (1,1), \quad (a_4,b_4) = (1,1).
\]
Using the arithmetical structure relations
\[
d_{1,i} = \frac{a_{i-1} + a_{i+1} + b_i}{a_i}, \qquad
d_{2,i} = \frac{b_{i-1} + b_{i+1} + a_i}{b_i},
\]
we compute:
\[
(d_{1,1}, d_{2,1}) = (1,3), \quad
(d_{1,2}, d_{2,2}) = (4,3), \quad
(d_{1,3}, d_{2,3}) = (4,3), \quad
(d_{1,4}, d_{2,4}) = (2,2).
\]
Thus,
$$
{\bf r} = (2,1,1,1,1,1,1,1), \,\,\,\,
{\bf d} = (1,4,4,2,3,3,3,2), $$
and $\gcd({\bf r})=1$. Since $a_1 \neq b_1$, the resulting arithmetical structure is non-symmetric.
\end{exam}

\begin{remark}
The converse of the Theorem~\ref{nonsym} also holds: if $(d,r)$ is an arithmetical structure on $P_2 \square P_m$, then the sequences $(a_i)$ and $(b_i)$ obtained from $r$ satisfy
\[
a_i \mid (a_{i-1} + a_{i+1} + b_i), \qquad
b_i \mid (b_{i-1} + b_{i+1} + a_i)
\]
for all $i=1,\dots,m$, where we adopt the boundary convention $a_0 = b_0 = a_{m+1} = b_{m+1} = 0$. This follows directly from the Laplacian equations
\[
d(v)\,r(v) = \sum_{u \sim v} r(u),
\]
and hence the theorem provides a complete characterization of arithmetical structures on $P_2 \square P_m$.
\end{remark}

Note that the graph $P_2 \square P_m$ can be viewed as a chain of $m-1$ copies of the cycle graph $C_4$, glued along common edges. This perspective allows one to interpret arithmetical structures on $P_2 \square P_m$ as arising from compatible assignments on consecutive $C_4$ subgraphs. In particular, local constraints on each $C_4$ must agree on shared vertices, leading to a global system of relations.

Moreover, imposing the symmetry condition $r(1,i)=r(2,i)$ reduces the problem to arithmetical structures on $P_m$, yielding a subclass enumerated by the Catalan numbers. However, the presence of multiple $C_4$ components introduces additional degrees of freedom, giving rise to non-symmetric arithmetical structures beyond this subclass.

\section{Arithmetical Structures on $P_2 \square P_m$ via $C_4$-Decomposition}
We view the graph $P_2 \square P_m$ as a chain of $m-1$ copies of $C_4$, where consecutive copies share a common edge.

\begin{lemma}
The graph $P_2 \square P_m$ can be decomposed into $m-1$ subgraphs, each isomorphic to $C_4$, such that consecutive $C_4$'s share exactly one edge.
\end{lemma}

\begin{proof}
Label the vertices of $P_2 \square P_m$ as $(1,i)$ and $(2,i)$ for $i=1,\dots,m$. 
For each $i=1,\dots,m-1$, the vertices
\[
(1,i), (1,i+1), (2,i+1), (2,i)
\]
form a 4-cycle. These cycles share the edge between $(1,i+1)$ and $(2,i+1)$ with the next cycle, forming a chain.
\end{proof}



\begin{proposition}
Let $(\mathbf{d},\mathbf{r})$ be an arithmetical structure on $C_4$ and $ d_1r_1 = d_2r_2$. Then the opposite vertex labels satisfy $r_1 + r_3 = r_2 + r_4.$
\end{proposition}

\begin{proof}
For $C_4$, the Laplacian equations are
\[
d_1r_1=r_2+r_4,\qquad d_2r_2=r_1+r_3,
\]
\[
d_3r_3=r_2+r_4,\qquad d_4r_4=r_1+r_3.
\]
In particular, $
d_1r_1=d_3r_3
\quad\text{and}\quad
d_2r_2=d_4r_4.$
If $d_1r_1=d_2r_2$, then using the vertex equations at $1$ and $2$, we obtain $
r_2+r_4=d_1r_1=d_2r_2=r_1+r_3.$
Hence,
\[
r_1+r_3=r_2+r_4.
\]
\end{proof}

\begin{corollary}
Let $(d,r)$ be an arithmetical structure on $P_2 \square P_m$ such that $
r(1,i) = r(2,i) \quad \text{for all } i.$
Then $(d,r)$ is in bijection with an arithmetical structure on $P_m$. 
In particular, the number of such symmetric structures is the Catalan number $C_{m-1}$.
\end{corollary}

\begin{proof}
Set $x_i = r(1,i) = r(2,i).$
Then the defining equations of the arithmetical structure give
\[
d(1,i)\,x_i = x_{i-1} + x_{i+1} + x_i.
\]
Rewriting, we obtain
\[
(d(1,i) - 1)x_i = x_{i-1} + x_{i+1}.
\]
Let $d'_i = d(1,i) - 1$. Then
\[
d'_i x_i = x_{i-1} + x_{i+1},
\]
which is precisely the defining condition for an arithmetical structure on $P_m$. 

Conversely, given an arithmetical structure $(d', x)$ on $P_m$, define
\[
r(1,i) = r(2,i) = x_i, \quad d(1,i) = d(2,i) = d'_i + 1.
\]
Then this defines an arithmetical structure on $P_2 \square P_m$. Hence, the correspondence is bijective.
\end{proof}

We encode arithmetical structures on $P_2 \square P_m$ as admissible sequences of states and apply a transfer-matrix method to obtain a new framework for enumerating arithmetical structures on ladder graphs.

\subsection{A Tranisition/Compatibility-Matrix Approach with Integer States}
Let $G = P_n \square P_m$, where $P_n$ and $P_m$ are path graphs.
We view $G$ as an $n \times m$ grid with vertices $(i,j)$, where
$1 \le i \le n$ and $1 \le j \le m.$

Each pair of consecutive columns $j$ and $j+1$ induces a $4$-cycle
\[
(i,j),\ (i,j+1),\ (i+1,j+1),\ (i+1,j),
\]
for $i = 1, \dots, n-1$, which is isomorphic to $(C_4)$.
Thus, any arithmetical structure on $G$ must restrict to an arithmetical structure on each such $C_4$.

The Laplacian conditions defining an arithmetical structure on $P_2 \square P_m$ involve only adjacent columns. This locality suggests encoding each column by a state and studying how these states evolve along the graph. By interpreting the compatibility between consecutive columns as admissible transitions, arithmetical structures can be viewed as sequences of states satisfying local constraints. This leads naturally to a description in terms of walks in a finite directed graph.

\begin{definition}[State space]
For each $1 \le i \le m$, define the state
$$S_i = (r_{1,i}, r_{2,i}) \in \mathbb{Z}_{>0}^2 \text{   such that        }\gcd(r_{1,i}, r_{2,i}) = 1.
$$
\end{definition}

\begin{definition}[Admissible transitions]
Let $S$ be the set of all primitive pairs that occur as $(r_{1,i}, r_{2,i})$ in some arithmetical structure on $C_4$. Let $(a, b), (c, d) \in S$. We say that $(a,b)$ is  admissible  (or compatible) to $(c,d)$  if there exists an arithmetical structure on the 4-cycle with vertices $
(1,i), (1,i+1), (2,i+1), (2,i)$
(in cyclic order) such that $$(r_{1,i}, r_{2,i}) = (a,b), \quad (r_{1,i+1}, r_{2,i+1}) = (c,d).$$ 
\end{definition}


\begin{definition}[Transition Matrix/Compatibility matrix]
Define the matrix $T = (t_{(a,b),(c,d)})$ indexed by $S$ by
\[
t_{(a,b),(c,d)} =
\begin{cases}
1, & \text{if } (a, b) \rightarrow (c, d) \text{   is admissible}, \\
0, & \text{otherwise}.
\end{cases}
\]
Then $T^{m-1}$ counts all valid walks of length $m-1$.
\end{definition}


We associate to the state space $S$ a directed graph whose vertex set is $S$,  and where there is a directed edge from $(a,b)\in S$ to $(c,d)\in S$  if and only if the transition $(a,b)\leftrightarrow(c,d)$ is admissible, i.e., they can occur as consecutive columns in an arithmetical structure on the $4$-cycle. By construction, the transition matrix $T$ is the adjacency matrix of this directed graph. For example, 


\begin{exam}
Consider the ladder graph $P_2 \square P_4$. Each column $i$ ($1 \le i \le 4$) is represented by a state $
S_i = (r(1,i), r(2,i)).$
For illustration, consider two consecutive states $
S_2 = (r_{12}, r_{22})= (2,1), \, S_3 =(r_{13}, r_{23})= (1,2).$ 
To determine whether the transition $S_2 \rightarrow S_3$ is admissible, we consider the $4$-cycle formed by the vertices $(1,2), (2, 2), (1,3), (2,3).$
The Laplacian conditions on this cycle impose compatibility relations between the values in the two columns. If these conditions are satisfied, then the transition is allowed. 
In this case, we draw a directed edge from $S_2= (2, 1)$ to $S_3=(1, 2)$ in the state graph. Repeating this process for all possible pairs of states in $P_2 \square P_4$ produces a finite directed graph, where
\begin{itemize}
\item vertices are all admissible states,
\item directed edges represent valid transitions between consecutive columns.
\end{itemize}
An arithmetical structure on $P_2 \square P_4$ is then equivalent to a walk of length $3$ in this directed graph.
\end{exam}

\begin{theorem}\label{tmthm}
Let $G = P_2 \square P_m$, and let $S$ be a finite set of admissible states, where each state is $S_i = (r(1,i), r(2,i))$. Let $r(1,i)| r(1,i+1)$ and $r(2,i)| r(2,i+1)$, for $1<i< m$. Then the arithmetical structures on $G$ are in one-to-one correspondence with admissible sequences of states $(S_1, S_2, \dots, S_m)$, where each $S_i \in S$ and each transition $S_i \to S_{i+1}$ is admissible.

Consequently, the number of arithmetical structures on $G$ is equal to the number of walks of length $m-1$ in the directed graph defined by the transition matrix $T$. If every state can occur as both an initial and terminal state, then this number is given by
$$
\mathbf{1}^T T^{m-1} \mathbf{1},$$ where $\mathbf{1}$ is the all-ones vector.
\end{theorem}

\begin{proof}
Each arithmetical structure assigns a positive integer value $r(v)$ to every vertex $v$ of $G$. For each column $i$, define the state $S_i = (r(1,i), r(2,i))$.

The Laplacian condition at each vertex relates values in column $i$ with those in adjacent columns $i-1$ and $i+1$. Given that  $r(1,i)| r(1,i+1)$ and $r(2,i)| r(2,i+1)$, for $1<i< m$. So, for each pair of consecutive columns $i$ and $i+1$, these conditions impose compatibility relations between $S_i$ and $S_{i+1}$. By definition, this means that $S_i \to S_{i+1}$ is an admissible transition. Therefore, any global arithmetical structure determines a sequence $(S_1, S_2, \dots, S_m)$ in which every consecutive pair satisfies the admissibility condition.

Conversely, suppose we are given a sequence $(S_1, S_2, \dots, S_m)$ such that each
transition $S_i \to S_{i+1}$ is admissible. We define a labeling $r$ on the vertices of $G$ by assigning the entries of $S_i$
to the vertices in column $i$. This is well-defined since each vertex belongs to
exactly one column.
We now verify that $r$ satisfies the Laplacian condition at every vertex.

We now verify that $r$ satisfies the Laplacian condition at every vertex. 
Let $(a,i)$ be an interior vertex with $1 < i < m$. The Laplacian condition at
$(a,i)$ involves contributions from the vertex itself, its vertical neighbor in the
same column, and its horizontal neighbors in columns $i-1$ and $i+1$.
The admissibility of the transition $S_{i-1} \to S_i$  and  $r(1,i)| r(1,i+1)$ and $r(2,i)| r(2,i+1)$, for $1<i< m$ guaranties that all relations involving columns $i-1$ and $i$ are satisfied, while the admissibility of $S_i \to S_{i+1}$ guaranties all relations involving columns $i$ and $i+1$.
Since every edge of $G$ is contained in exactly one or two adjacent column pairs, and the labeling is consistent on shared vertices, these local conditions combine
to give the full Laplacian equation at $(a,i)$. For boundary vertices ($i=1$ or $i=m$), only one adjacent transition  is involved,
and the admissibility condition for that transition ensures the required Laplacian relations. 

Since every Laplacian equation involves only vertices within one column or adjacent columns, all conditions are satisfied, and hence, all vertices satisfy $(\mathrm{diag}(d) - A)r = 0$, so $r$ is an arithmetical structure on $G$.
Thus, arithmetical structures on $G$ are in bijection with admissible sequences of length $m$.

Finally, such sequences correspond exactly to walks of length $m-1$ with the adjacency matrix is $T$. Therefore, the number of arithmetical structures equals the number of such walks. If every state in $S$ can serve as both an initial and terminal state, this number is given by $\mathbf{1}^T T^{m-1} \mathbf{1}$.
\end{proof}

\begin{remark}
We have the following consequences:
\begin{itemize}
\item The counting problem reduces to a transfer-matrix method, where all local constraints arise from $C_4$.
\item The growth of $N_{m}$ is asymptotically governed by the spectral radius $\rho(T)$ of the transition matrix: $N_{m} \sim \rho(T)^{\,m-1}$.
\item For $n=2$, the state space reduces to pairs $(a,b)$, and admissibility is determined directly by $C_4$, giving an explicit finite matrix $T$.
\item This framework extends naturally to higher-dimensional grids and other Cartesian products, highlighting the generality of the approach.
\end{itemize}
\end{remark}

The algorithm provided below gives a finite-state dynamic programming approach to counting arithmetical structures.
\begin{algorithm}[H]
\caption{Transition Matrix Method for $P_2 \square P_m$}
\begin{enumerate}
    \item Compute all primitive arithmetical structures on $C_4$, i.e., find all pairs $(d,r)$ such that
    \[
    \operatorname{diag}(d)\,r = A(C_4)\,r,\quad r \in \mathbb{Z}_{>0}^4,\quad \gcd(r)=1.
    \]
    \item Extract the state space $S$, consisting of all primitive pairs $(a,b)$ that appear as $(r_{1,i}, r_{2,i})$ in these structures such that  $r(1,i)| r(1,i+1)$ and $r(2,i)| r(2,i+1)$, $1< i<m$.
    \item For each $(a,b),(c,d) \in S$, declare $(a,b) \to (c,d)$ admissible if there exists a primitive arithmetical structure on $C_4$ whose first column is $(a,b)$ and second column is $(c,d)$.
    \item Construct the transition matrix $T = (T_{(a,b),(c,d)})$ indexed by $S$, where
    \[
    T_{(a,b),(c,d)} =
    \begin{cases}
    1, & \text{if } (a,b) \to (c,d) \text{ is admissible}, \\
    0, & \text{otherwise}.
    \end{cases}
    \]

    \item Compute $T^{m-1}$.

    \item Compute
    \[
    N_m = \mathbf{1}^T T^{m-1} \mathbf{1},
    \]
    where $\mathbf{1}$ is the all-ones vector indexed by $S$.
\end{enumerate}
\end{algorithm}

\begin{remark}
Note that $\mathbf{1}^T T^{\,m-1} \mathbf{1}$ counts all admissible sequences of length $m$. Further, $$
    N_m = \text{number of valid sequences } (S_1, \dots, S_m)
= \text{number of arithmetical structures on } P_2 \square P_m.$$
   \begin{itemize}
    \item Each column $i$ of the ladder graph corresponds to a state $S_i$,
    \item Valid structures correspond to walks in a finite state graph,
    \item $T$ is the adjacency matrix of that state graph,
    \item Hence, $N_m$ counts all walks of length $m-1$.
\end{itemize} 
\end{remark}

\subsection{Construction of the Transition Matrix for $P_2 \square P_m$}
\begin{enumerate}
 \item (State Space) For $G = P_2 \square P_m$, each column defines a state $
S_i = (r_{1,i}, r_{2,i}) \in \mathbb{Z}_{>0}^2,$
with $\gcd(r_{1,i}, r_{2,i}) = 1$.
Let $S$ denote the set of all primitive pairs $(a,b)$ that arise as $(r_{1,i}, r_{2,i})$ in some arithmetical structure on $C_4$.

\item (Admissibility Condition) Two states $(a,b)$ and $(c,d)$ in $S$ are said to be admissible if there exists an arithmetical structure on the 4-cycle $
(1,i), (1,i+1), (2,i+1), (2,i)$
(in cyclic order) such that $
(r_{1,i}, r_{2,i}) = (a,b), \quad (r_{1,i+1}, r_{2,i+1}) = (c,d).$

\item {Arithmetical Structure on $C_4$}
Let $r = (a,c,d,b)$ correspond to the vertices $
(1,i), (1,i+1), (2,i+1), (2,i).$
Then the defining equations $
\operatorname{diag}(d)\,r = A(C_4)\,r$ are equivalent to
\[
\begin{aligned}
d_1 a &= b + c, \\
d_2 c &= a + d, \\
d_3 d &= b + c, \\
d_4 b &= a + d.
\end{aligned}
\]

\item
The above system admits a solution in positive integers $d_1,d_2,d_3,d_4$ if and only if
\[
a \mid (b + c), \quad b \mid (a + d), \quad c \mid (a + d), \quad d \mid (b + c).
\]

Thus, admissibility of $(a,b)$ and $(c,d)$ is equivalent to the above divisibility conditions.

\item Define the transition matrix $T = (t_{(a,b),(c,d)})$ indexed by $S$ by
\[
t_{(a,b),(c,d)} =
\begin{cases}
1, & \text{if } a \mid (b + c),\; b \mid (a + d),\; c \mid (a + d),\; d \mid (b + c), \\
0, & \text{otherwise}.
\end{cases}
\]

\item{Algorithm}

\begin{enumerate}
    \item Compute all primitive arithmetical structures on $C_4$.
    \item Extract the state space $S$ of primitive pairs $(a,b)$.
    \item For each $(a,b),(c,d) \in S$, check the divisibility conditions:
    \[
    a \mid (b + c), \quad b \mid (a + d), \quad c \mid (a + d), \quad d \mid (b + c).
    \]
    \item Construct the transition matrix $T$.
    \item Compute $T^{m-1}$.
    \item Compute $N_m = \mathbf{1}^T T^{m-1} \mathbf{1},$
    where $\mathbf{1}$ is the all-ones vector indexed by $S$.
\end{enumerate}
\end{enumerate}

\begin{exam}
Label the vertices of $C_4$ as $(1,i), (1,i+1), (2,i+1), (2,i)$ in cyclic order. Let $r = (a, c, d, b),$
where $(a,b)$ corresponds to column $i$ and $(c,d)$ corresponds to column $i+1$.
The defining equations of an arithmetical structure,
\[
\operatorname{diag}(d)\,r = A(C_4)\,r,
\]
yield the system
\[
\begin{aligned}
d_1 a &= c + b, \\
d_2 c &= a + d, \\
d_3 d &= c + b, \\
d_4 b &= a + d.
\end{aligned}
\]
We require $a,b,c,d > 0$ and $\gcd(a,c,d,b) = 1$. Solving this system in positive integers yields all primitive arithmetical structures on $C_4$.

For $C_4$, the state space is the set of primitive column pairs that arise is
$$S = \{(2,1), (1,1), (1,2), (3,2)\}.$$
Now,  the admissible transitions between states are given by
$$
\begin{aligned}
(2,1) & \to (1,1), (1,2),\\
(1,1) &\to (1, 1),(1,2), (3,2), \\
(1,2) &\to (1,1), (2,1),\\
(3,2) &\to (1,1).
\end{aligned}$$
Ordering the states as $(2,1),(1,1), (1,2), (3,2)$, the transition matrix is given by
$$
T =
\begin{pmatrix}
0 & 1 & 1 & 0 \\
1 & 1 & 1 & 1 \\
1 & 1 & 0 & 0 \\
0 & 1 & 0 & 0
\end{pmatrix}.$$
For example, consider $m=3$. 
We compute
$$
T^2 =
\begin{pmatrix}
2 & 2 & 1 & 1\\
2 & 4 & 2 & 1 \\
1 & 2 & 2 & 1 \\
1 & 1 & 1 & 1 
\end{pmatrix}.
$$
Thus, we have
$$N_{3} = \mathbf{1}^T T^2 \mathbf{1}
= (1,1,1)
\begin{pmatrix}
2 & 2 & 1 & 1\\
2 & 4 & 2 & 1 \\
1 & 2 & 2 & 1 \\
1 & 1 & 1 & 1 
\end{pmatrix}
\begin{pmatrix}
1 \\ 1 \\ 1
\end{pmatrix}
= 25.$$

Now we have state space $ = \{S_1,S_2, S_3, S_4\}=\{(2,1),(1,1),(1,2),(3,2)\}$ and  with respect to column ordering adjacency matrix, arithmetical structures for  $P_2 \square P_3$ is given by
\begin{table}[ht]
\centering
\begin{tabular}{|c|c|c|c|}
\hline
No. & State Sequence & $r$-vector  & $d$-vector\\ \hline
1 & $(S_1,S_2,S_2)$ & $(2,1,1,1,1,1)$ & (1,3,4,3,2,2)\\ \hline
2 & $(S_1,S_2,S_3)$ & $(2,1,1,1,1,2)$ & (1,3,4,4,3,1) \\ \hline
3 & $(S_1, S_2, S_4)$ &$ (2,1,1,1,3,2)$ & (1,3,6,4,1,2)  \\ \hline
4 & $(S_1,S_2,S_1)$ & $(2,1,1,1,2,1)$ & (1,3,5,3,1,3)\\ \hline
5 & $(S_1, S_3, S_1)$ & $(2,1,1,1,1,1)$ & (1,3,4,3,2,2)  \\ \hline
6 & $(S_2,S_2,S_2)$ & $(1,1,1,1,1,1)$ & (2,2,3,3,2,2) \\ \hline
7 & $(S_2,S_2,S_3)$ & $(1,1,1,1,1,2)$  & (2,2,3,4,3,1)\\ \hline
8 & $(S_2,S_2,S_4)$ & $(1,1,1,1,3,2)$ & (2,2,5,4,1,2)\\ \hline
9 & $(S_2,S_2,S_1)$ & $(1,1,1,1,2,1)$ & (2,2,4,3,1,3) \\ \hline
10 & $(S_3,S_2,S_2)$ & $(1,2,1,1,1,1)$  & (3,1,3,4,2,2)\\ \hline
11 & $(S_3,S_2,S_3)$ & $(1,2,1,1,1,2)$ & (3,1,3,5,3,1)\\ \hline
12 & $(S_3,S_2,S_1)$ & $(1,2,1,1,2,1)$ & (3,1,4,4,1,3) \\ \hline
13 & $(S_3,S_2,S_4)$ & $(1,2,1,1,3,2)$ & (3,1,5,5,1,2)\\ \hline
14 & $(S_4,S_2,S_2)$ & $(3,2,1,1,1,1)$  & (1,2,5,4,2,2)\\ \hline
15 & $(S_4,S_2,S_3)$ & $(3,2,1,1,1,2)$ & (1,2,5,5,3,1)\\ \hline
16 & $(S_4,S_2,S_4)$ & $(3,2,1,1,3,2)$ & (1,2,7,5,1,2)\\ \hline
17 & $(S_4, S_2, S_1)$ &$ (3,2,1,1,2,1)$ &  (1,2,6,4,1,3) \\ \hline
\end{tabular}
\caption{Arithmetical structures on $P_2 \square P_3$ arising from the state space
$S=\{(2,1), (1,1),(1,2),(3,2)\}$}
\end{table}

\end{exam}

\begin{remark}
Let $(d,r)$ be an arithmetical structure on $P_2 \square P_m$ with
$r_{1,i} = a_i$ and $r_{2,i} = b_i$, with $a_i, b_i \in \mathbb{Z}_{>0}$ and define $\delta_i = a_i - b_i$. Then subtracting the equations coming from arithmetical structures on $P_2 \square P_m$ we have $
d_{1,i} a_i - d_{2,i} b_i = \delta_{i-1} + \delta_{i+1} - \delta_i,$ where $i=1:m$  and $a_0 = b_0 = a_{m+1} = b_{m+1} = 0$  In particular, if $d_{1,i}a_i = d_{2,i}b_i$
for all $i$, then $\delta_{i-1} + \delta_{i+1} = \delta_i$ for all $i$.
As a consequence, If for some $i$, $
a_i = b_i \quad \text{and} \quad a_{i+1} = b_{i+1},$
then $
a_j = b_j \quad \text{for all } j = 1, \dots, m.$ That is it says that Local symmetry at two consecutive positions forces global symmetry. So, if two consecutive columns are symmetric, then the entire ladder is symmetric.
\end{remark}

\subsection{Arithmetical Structures on $P_n\square P_m$}
The graph $P_m \square P_n$ is a rectangular grid with vertices $V = \{(i,j) : 1 \leq i \leq m,; 1 \leq j \leq n\}.$
Edges connect nearest neighbors horizontally and vertically.

\begin{definition}
Let $P_n$ and $P_m$ be path graphs on $n$ and $m$ vertices, respectively. 
The \emph{grid graph} is defined as the Cartesian product $
G = P_n \square P_m$ is a rectangular grid with vertex set 
\[
V(G) = \{(i,j) \mid 1 \le i \le n,\; 1 \le j \le m\}.
\]
Two vertices $(i,j)$ and $(k,\ell)$ are adjacent if and only if
\[
(i = k \text{ and } |j - \ell| = 1)
\quad \text{or} \quad
(j = \ell \text{ and } |i - k| = 1).
\]
\end{definition}
The adjacency matrix of the grid graph $ P_n \square P_m$ is given by
\[
A(G) = A(P_n) \otimes I_m \;+\; I_n \otimes A(P_m),
\]
where $A(P_n)$ and $A(P_m)$ are the adjacency matrices of the path graphs,
$I_n$ and $I_m$ are identity matrices.

Now, Consider $G= P_n \square P_m$. Let $({\bf d}, {\bf r})$  be an arithmetical structure on $G$. Then we have 
$$
\operatorname{diag}({\bf d}) {\bf r} =
\left(A(P_n)\otimes I_m + I_n \otimes A(P_m)\right) {\bf r}.
$$
Now, for each vertex $(i,j)$, for interior vertices  $(1<i<n,\; 1<j<m)$, our aim is to solve
\[
d_{i,j} r_{i,j} = r_{i-1,j} + r_{i+1,j} + r_{i,j-1} + r_{i,j+1}
\]
with boundary adjustments as follows: for corner vertices
\begin{align*}
d_{1,1} \, r_{1,1} &= r_{2,1} + r_{1,2}, \\
d_{1,m} \, r_{1,m} &= r_{2,m} + r_{1,m-1}, \\
d_{n,1} \, r_{n,1} &= r_{n-1,1} + r_{n,2}, \\
d_{n,m} \, r_{n,m} &= r_{n-1,m} + r_{n,m-1}, 
\end{align*}
For Boundary vertices we have
\begin{align*}
\text{Top row } (i=1,\; 1<j<m): \quad
d_{1,j} \, r_{1,j} &= r_{1,j-1} + r_{1,j+1} + r_{2,j}, \\
\text{Bottom row } (i=n,\; 1<j<m): \quad
d_{n,j} \, r_{n,j} &= r_{n,j-1} + r_{n,j+1} + r_{n-1,j}, \\
\text{Left column } (j=1,\; 1<i<n): \quad
d_{i,1} \, r_{i,1} &= r_{i-1,1} + r_{i+1,1} + r_{i,2}, \\
\text{Right column } (j=m,\; 1<i<n): \quad
d_{i,m} \, r_{i,m} &= r_{i-1,m} + r_{i+1,m} + r_{i,m-1}. 
\end{align*}
Technically, finding arithmetical structures on grid graphs is more involved. 
Therefore, in this paper, we present only a few results and touch upon the topic for the sake of completeness.

\begin{theorem}
Let $G = P_n \square P_m$. Suppose $({\bf d}, {\bf r})$ is an arithmetical structure on $G$ such that $
r_{i,j} = r_j \quad \text{for all } i = 1,\dots,n \text{ and } j = 1,\dots,m.$
Then the vector $(r_1, r_2, \dots, r_m)$ satisfies $
r_j \mid (r_{j-1} + r_{j+1})$
for all $j$ and hence determines an arithmetical structure on $P_m$.

Conversely, any arithmetical structure on $P_m$ induces such a column-constant arithmetical structure on $G$, where
$$
d_{i,j} = \frac{r_{j-1} + r_{j+1}}{r_j} + c_i,
\text{   with   } c_i = 2 \text{   for  } 1 < i < n \text{   and   } c_i = 1 \text{  for  } i=1,n.$$
\end{theorem}

\begin{proof}
Let $({\bf d}, {\bf r})$ be an arithmetical structure on $G$. Then
\[
(\mathrm{diag}(d) - A(G))r = 0 \implies 
\mathrm{diag}(d)\,r = A(G)\,r.
\]
Using the decomposition
\[
A(G) = A(P_n)\otimes I_m + I_n \otimes A(P_m),
\]
we obtain
\[
\mathrm{diag}(d)\,r = (A(P_n)\otimes I_m)r + (I_n \otimes A(P_m))r.
\]

Now assume $r_{i,j} = r_j$ for all $i$. Then $r$ consists of $n$ identical copies of the vector $(r_1,\dots,r_m)$.
Here and throughout, we adopt the convention that $r_0=0$
and $r_{m+1} = 0$, corresponding to the boundary vertices of $P_m$

First consider $(I_n \otimes A(P_m))r$. This acts along each row and gives, at each position $(i,j)$,
$r_{j-1} + r_{j+1}.$
Next, consider $(A(P_n)\otimes I_m)r$. This acts between rows. Since all entries in a column are equal, each vertical neighbor contributes $r_j$. Hence, at $(i,j)$, this term contributes $c_i r_j$, where $c_i = 2$ for $1<i<n$ and $c_i = 1$ for $i=1,n$.
Combining both contributions, we obtain
\[
d_{i,j} r_j = r_{j-1} + r_{j+1} + c_i r_j.
\]
Rewriting,
\[
d_{i,j} = \frac{r_{j-1} + r_{j+1}}{r_j} + c_i.
\]
Since $d_{i,j}$ is an integer, it follows that $
r_j \mid (r_{j-1} + r_{j+1}),$
so $(r_1,\dots,r_m)$ satisfies the arithmetical structure condition on $P_m$.

Conversely, suppose $(r_1,\dots,r_m)$ defines an arithmetical structure on $P_m$, so that
\[
\frac{r_{j-1} + r_{j+1}}{r_j}
\]
is a positive integer for all $j$. Define $r_{i,j} = r_j$ and set
\[
d_{i,j} = \frac{r_{j-1} + r_{j+1}}{r_j} + c_i.
\]
Then all entries of $d$ are positive integers, and $
(\mathrm{diag}(d) - A(G))r = 0$
holds by construction. Hence $({\bf d}, {\bf r})$ is an arithmetical structure on $G$.
\end{proof}

Let $G = P_n \square P_m$. Consider arithmetical structures $(\mathbf{d}, \mathbf{r})$ on $G$ satisfying
$r_{i,1} = r_{i,2} = \dots = r_{i,m} \quad \text{for all } i = 1, \dots, n.$
Then the above theorem gives such symmetric structures are in bijection with arithmetical structures on $P_n$.

Now, we derive the arithmetical structures on $P_n \square P_m$ via M-matrices.

\begin{theorem}
Let $P_n$ and $P_m$ be path graphs with $n$ and $m$ vertices, respectively.  
Suppose $(d_{P_n}, x)$ is an arithmetical structure on $P_n$ and $(d_{P_m}, y)$ is an arithmetical structure on $P_m$, with $x \in \mathbb{Z}_{>0}^n,  y \in \mathbb{Z}_{>0}^m, $
$$M_{P_n} x = (\mathrm{diag}(d_{P_n}) - A(P_n)) x = 0, \quad
M_{P_m} y = (\mathrm{diag}(d_{P_m}) - A(P_m)) y = 0.$$
Then an arithmetical structure $(d,r)$ on the Cartesian product $P_n \square P_m$ is given by
\[
{\bf r} = x \otimes y \in \mathbb{Z}_{>0}^{nm}, \quad
d(i,j) = d_{P_n,i} + d_{P_m,j} \in \mathbb{Z}_{>0}^{nm}.
\]
Equivalently, if $M_{P_n \square P_m} := \mathrm{diag}(d) - A(P_n \square P_m),$
then $
M_{P_n \square P_m}  {\bf r} = 0.$
\end{theorem}

\begin{proof}
The adjacency matrix of $P_n \square P_m$ satisfies
\[
A(P_n \square P_m) = A(P_n) \otimes I_m + I_n \otimes A(P_m),
\]
where $I_n, I_m$ are identity matrices.  
Consider $r = x \otimes y$ and $d(i,j) = d_{P_n,i} + d_{P_m,j}$. Using Kronecker product properties:
\[
(A(P_n) \otimes I_m)(x \otimes y) = (A(P_n) x) \otimes y = (d_{P_n} \circ x) \otimes y,
\]
\[
(I_n \otimes A(P_m))(x \otimes y) = x \otimes (A(P_m) y) = x \otimes (d_{P_m} \circ y),
\]
where $\circ$ denotes componentwise multiplication.  
Therefore,
\[
A(P_n \square P_m) \, r = (d_{P_n} \circ x) \otimes y + x \otimes (d_{P_m} \circ y) = \mathrm{diag}(d) \, r.
\]
It follows that
\[
M_{P_n \square P_m} \, r = (\mathrm{diag}(d) - A(P_n \square P_m)) \, r = 0.
\]
Since $x > 0$ and $y > 0$, we have $r > 0$, and by construction $d > 0$. Hence $(d,r)$ is an arithmetical structure on $P_n \square P_m$, as claimed.
\end{proof}

Now, we study arithmetical structures on the Cartesian product graph $P_n \square P_m$ via cycle graph $C_4$. 
This graph can be viewed as a grid and decomposed into copies of $C_4$, which provides a useful structural framework.

\begin{lemma}
The graph $P_n \square P_m$ can be decomposed into $(n-1)(m-1)$ subgraphs, each isomorphic to $C_4$, corresponding to the unit squares of the grid.
\end{lemma}

\begin{proof}
Label vertices as $(i,j)$ for $1 \le i \le n$, $1 \le j \le m$. 
Each set of vertices
\[
(i,j), (i+1,j), (i+1,j+1), (i,j+1)
\]
forms a 4-cycle. These cycles cover the grid and intersect along edges.
\end{proof}

\begin{definition}[State Space]
Let $G=P_n \square P_m$, and let $(\mathbf d,\mathbf r)$ be an arithmetical
structure on $G$. For each $1\le i\le m$, define the \emph{state} of the
$i$th column by
$
S_i=(r_{1,i},r_{2,i},\dots,r_{n,i})\in \mathbb Z_{>0}^n,$
where $
\gcd(r_{1,i},r_{2,i},\dots,r_{n,i})=1.$
The set of all such primitive $n$-tuples arising from arithmetical structures
on $P_n \square P_m$ is called the \emph{state space} and is denoted by $S$.
\end{definition}

The graph $P_n \square P_m$ can be decomposed into $(n-1)(m-1)$ subgraphs,
each isomorphic to $C_4$, corresponding to the unit squares of the grid.
In particular, the subgraph induced by any two consecutive columns
$i$ and $i+1$ is isomorphic to $P_n \square P_2$, which consists of
$n-1$ copies of $C_4$ arranged in a chain.

\begin{definition}[Admissible Transitions]
Let $\mathbf a=(a_1,\dots,a_n)$ and
$\mathbf b=(b_1,\dots,b_n)$ be two states in $S$.
We say that $\mathbf a$ is \emph{admissible} (or \emph{compatible}) to
$\mathbf b$, written $\mathbf a \to \mathbf b,$
if there exists an arithmetical structure on $P_n \square P_2$ such that
the first column has state $\mathbf a$ and the second column has state
$\mathbf b$.

Equivalently, for each $1\le k\le n-1$, the restriction to the square
with vertices
\[
(k,i),\ (k,i+1),\ (k+1,i+1),\ (k+1,i)
\]
forms an arithmetical structure on a copy of $C_4$, and these
$n-1$ local structures agree on their shared edges.
\end{definition}

\begin{definition}[Transition Matrix]
Let $S$ be the state space. The \emph{transition matrix}
(or \emph{compatibility matrix}) is the matrix $
T=(t_{\mathbf a,\mathbf b})_{\mathbf a,\mathbf b\in S},$
defined by
\[
t_{\mathbf a,\mathbf b}=
\begin{cases}
1, & \text{if } \mathbf a\to \mathbf b \text{ is admissible},\\
0, & \text{otherwise}.
\end{cases}
\]
\end{definition}


\begin{theorem}\label{thm:grid-walk-correspondence}
Let $G=P_n\square P_m$, and let $\mathcal{S}$ be the finite state space of admissible
column states $S_i = (r_{1,i}, r_{2,i}, \dots, r_{n,i}) \in \mathbb{Z}_{>0}^n$. Let $ r(k,i)|r(k,i + 1), 1<k<n, 1<i<m$. Then arithmetical structures on $G$ are in one-to-one correspondence
with admissible sequences $
(S_1,S_2,\dots,S_m),$
where each $S_i\in \mathcal{S}$ and each transition $
S_i\to S_{i+1}, \qquad 1\le i\le m-1,$
is admissible.

Consequently, the number of arithmetical structures on $G$ is equal to the number
of walks of length $m-1$ in the directed transition graph on $\mathcal{S}$.
If every state can occur as both an initial state and a terminal state, then
\[
N_{n,m}=\mathbf{1}^{T}T^{\,m-1}\mathbf{1},
\]
where $T$ is the transition matrix of $\mathcal{S}$ and $\mathbf{1}$ is the all-ones
vector indexed by $\mathcal{S}$.
\end{theorem}

\begin{proof}
Proof follows from the proof of the Theorem~\ref{tmthm}. 
\end{proof}

\section{Conclusion}
In this paper, we studied arithmetical structures on Cartesian product graphs, 
with a particular focus on ladder graphs $P_2 \square P_m$ and grid graphs 
$P_n \square P_m$. Starting from the definition of an arithmetical structure 
$(\mathbf{d}, \mathbf{r})$, we analyzed how such structures behave under 
graph products and how they relate to known results for simpler graph families 
such as paths and cycles.

For ladder graphs, we established structural properties of 
$\operatorname{Arith}(P_2 \square P_m)$ and identified recurring patterns 
that allow these structures to be understood in terms of arithmetical 
structures on $P_m$. We also explored an alternative perspective via 
$C_4$-decomposition, which provides a more refined understanding of the 
local behavior of these structures.

We then extended our study to grid graphs $P_n \square P_m$, where the 
interaction between multiple dimensions introduces additional complexity. 
Despite this, we observed that several structural features from the ladder 
case persist, offering insight into the characterization and potential 
enumeration of arithmetical structures in higher dimensions.

Overall, this work contributes to the growing theory of arithmetical 
structures by extending it to product graphs and highlighting new 
combinatorial phenomena. These results open several directions for 
future research, including explicit enumeration formulas for 
$\operatorname{Arith}(P_n \square P_m)$, deeper connections with 
Laplacian-based invariants, and the study of arithmetical structures 
on more general classes of graph products.

\vspace{0.5cm}
\noindent{\bf Compliance with Ethical Standards.}
\begin{itemize}
\item {\bf Data Availability Statement:} Our manuscript has no associated data. 
\item {\bf Conflict of Interests:} The authors declare no conflict of interest.\\
\end{itemize}

\bibliographystyle{plain}
\bibliography{main}

@book{Berman1994andPlemmons,
  title={Nonnegative matrices in the mathematical sciences},
  author={Berman, Abraham and Plemmons, Robert J},
  year={1994},
  publisher={SIAM}
}

@article{BHDNJC18,
   title={Counting arithmetical structures on paths and cycles},
   volume={341},
   ISSN={0012-365X},
   url={http://dx.doi.org/10.1016/j.disc.2018.07.002},
   DOI={10.1016/j.disc.2018.07.002},
   number={10},
   journal={Discrete Mathematics},
   publisher={Elsevier BV},
   author={Braun, Benjamin and Corrales, Hugo and Corry, Scott and Puente, Luis David García and Glass, Darren and Kaplan, Nathan and L. Martin, Jeremy and Musiker, Gregg and Valencia, Carlos E.},
   year={2018},
 pages={2949--2963} 
}

@article{CHVC18,
  title={Arithmetical structures on graphs with connectivity one},
  author={Corrales, Hugo and Valencia, Carlos E},
  journal={Journal of Algebra and its Applications},
  volume={17},
  number={08},
  year={2018},
  publisher={World Scientific}
}

@misc{CP18,
  title={Divisors and Sandpiles: An Introduction to Chip-Firing},
  author={Glass, Darren},
  year={2020},
  publisher={JSTOR}
}

@article{HCEV18,
  title={Arithmetical structures on graphs},
  author={Corrales, Hugo and Valencia, Carlos E},
  journal={Linear Algebra and its Applications},
  volume={536},
  pages={120--151},
  year={2018},
  publisher={Elsevier}
}

@article{KAAL20,
  title={Arithmetical structures on bidents},
  author={Archer, Kassie and Bishop, Abigail C and Diaz-Lopez, Alexander and Puente, Luis D Garcia and Glass, Darren and Louwsma, Joel},
  journal={Discrete Mathematics},
  volume={343},
  number={7},
  pages={1--23},
  year={2020},
  publisher={Elsevier}
}

@article{LD89,
  title={Arithmetical graphs},
  author={Lorenzini, Dino J},
  journal={Mathematische Annalen},
  volume={285},
  number={3},
  pages={481--501},
  year={1989}
}

@article{criticalpolynomialofgraphbyLorenzini,
  title={The critical polynomial of a graph},
  author={Lorenzini, Dino},
  journal={Journal of Number Theory},
  volume={257},
  pages={215--248},
  year={2024},
  publisher={Elsevier}
}

@article{Corry2018DivisorsAS,
  title={Divisors and Sandpiles: An Introduction to Chip-Firing},
  author={Scott Corry and David Perkinson},
journal={AMS Non-Series Monographs},
volume={114},
  year={2018},
  url={https://api.semanticscholar.org/CorpusID:240373917}
}

@article{AAJL18,
 title={Critical groups of arithmetical structures under a generalized star-clique operation},
  author={Alexander Diaz-Lopez and Joel Louwsma},
journal={Linear Algebra and its Applications},
volume={656},
  year={2023},
 pages={324--344}
}

@article{BANBDCRBY2024,
 title={Arithmetical Structures on Wheel Graphs},
  author={Adhikari, Bibhas and  Behera, Namita and Ram Chhetri, Dilli and  Yadav, Raj Bhawan },
journal={https://arxiv.org/abs/2412.07816},
  year={2024}
}

@article{arithmeticaloncompletegraphs,
  title={On arithmetical structures on complete graphs},
  author={Harris Zachary and Louwsma Joel},
  journal={Involve, a Journal of Mathematics},
  volume={13},
  number={2},
  pages={345--355},
  year={2020},
  publisher={Mathematical Sciences Publishers}
}

@article{catalannumberbycarlitz,
  title={Sequences, paths, ballot numbers},
  author={Carlitz, L},
  journal={The Fibonacci Quarterly},
  volume={10},
  number={5},
  pages={531--549},
  year={1972},
  publisher={Taylor $\&$ Francis}
}

@book{west2001introduction,
  title={Introduction to graph theory},
  author={West, Douglas Brent and others},
  volume={2},
  year={2001},
  publisher={Prentice hall Upper Saddle River}
}

@article{lionelleivnewhatisasandpile,
  title={Simplicial and cellular trees},
  author={Duval, Art M and Klivans, Caroline J and Martin, Jeremy L},
  journal={arXiv preprint arXiv:1506.06819},
  year={2015}
}

@article{keyesReiter2021boundingthenumberofArithstructure,
  title={Bounding the number of arithmetical structures on graphs},
  author={Keyes, Christopher and Reiter, Tomer},
  journal={Discrete Mathematics},
  volume={344},
  number={9},
  pages={1--11},
  year={2021},
  publisher={Elsevier}
}

@article{PathwithDoubleedgeArithstructure,
  title={Arithmetical structures on paths with a doubled edge},
  author={Glass, Darren and Wagner, Joshua},
  journal={arXiv preprint arXiv:1903.01398},
  year={2019}
}

@article{Chipfiringgamesongraph1991,
  title={Chip-firing games on graphs},
  author={Bj{\"o}rner, Anders and Lov{\'a}sz, L{\'a}szl{\'o} and Shor, Peter W},
  journal={European Journal of Combinatorics},
  volume={12},
  number={4},
  pages={283--291},
  year={1991},
  publisher={Elsevier}
}

@article{SandpilegroupoverEllpiticCurves,
  title={The critical groups of a family of graphs and elliptic curves over finite fields},
  author={Musiker, Gregg},
  journal={Journal of Algebraic Combinatorics},
  volume={30},
  number={2},
  pages={255--276},
  year={2009},
  publisher={Springer}
}

@article{RMerris94,
title={Laplacian matrices of graphs: A survey},
author={R. Merris}, 
journal={Linear Algebra Appl.}, 
volume={197/198 },
pages= {143--176},
year={1994}
}

@article{biggs1999,
   title={Chip-firing and the critical group of a graph},
   author={Biggs, Norman},
   journal={Journal of Algebraic Combinatorics},
   volume={9},
   pages={25--45},
   year={1999},
   doi={10.1023/A:1008636923550}
 }

@article{archer2025,
   title={Generalized chip firing and critical groups of arithmetical structures on trees},
   author={Archer, Kevin and Diaz-Lopez, Alejandro and Glass, Daniel and Louwsma, J.},
   journal={arXiv preprint arXiv:2505.05392},
   year={2025},
   note={\url{https://arxiv.org/abs/2505.05392}}
 }

@article{archer2023,
   title={Critical groups of arithmetical structures on star and complete graphs},
   author={Archer, Kevin and Diaz-Lopez, Alejandro and Glass, Daniel and Louwsma, J.},
  journal={arXiv preprint arXiv:2301.02114},
   year={2023},
   note={\url{https://arxiv.org/abs/2301.02114}}
 }

@article{braun2017arithmetical,
   title={Arithmetical structures on paths and cycles},
   author={Braun, Benjamin and Davis, Robert},
   journal={Journal of Combinatorial Theory, Series A},
   volume={150},
   pages={1--23},
   year={2017},
   publisher={Elsevier}
 }

@book{imrich2000product,
   title={Product Graphs: Structure and Recognition},
   author={Imrich, Wilfried and Klav{\v{z}}ar, Sandi},
   year={2000},
    publisher={Wiley}
 }

@article{chhetri2025arithmeticalonFanGraph,
  title={Arithmetical Structures On Fan Graphs},
  author={Chhetri, Dilli Ram and Behera, Namita and Yadav, Raj Bhawan},
  journal={arXiv preprint arXiv:2503.01913},
  year={2025}
}

@book{CartesianProductofGraphs,
  title={Topics in graph theory: Graphs and their Cartesian product},
  author={Imrich, Wilfried and Klavzar, Sandi and Rall, Douglas F},
  year={2008},
  publisher={CRC Press}
}

@book{Harary,
  title={Graph Theory},
  author={F. Harary},
  year={1988},
  publisher={Narosa Publishing House}
}

\end{document}